\documentclass[11pt]{amsart}

\usepackage{euscript}
\usepackage{amsmath}
\usepackage{amsfonts}
\usepackage{amssymb}
\usepackage{mathrsfs}
\usepackage{amsthm}
\usepackage{epsfig}
\usepackage{epstopdf}
\usepackage{url}
\usepackage{array}
\usepackage{amssymb}\usepackage{amscd}
\usepackage{epic}
\usepackage{eufrak}
\usepackage{tikz,underscore}
\usepackage{tikz-cd}
\usepackage{float}
\usepackage{subfigure}
\usepackage{tkz-euclide}
\usepackage{yhmath}
\usepackage{tensor}

\usepackage{graphicx}

\numberwithin{equation}{section}

\theoremstyle{plain}
\newtheorem{theorem}{Theorem}[section]
\newtheorem{lemma}[theorem]{Lemma}
\newtheorem{proposition}[theorem]{Proposition}

\newtheorem{corollary}[theorem]{Corollary}

\newtheorem{conjecture}[theorem]{Conjecture}

\newtheorem{cor}[equation]{Corollary}

\newtheorem{prop}[equation]{Proposition}

\newtheorem*{claim*}{Claim}

\theoremstyle{definition}

\newtheorem{remark}[theorem]{Remark}

\DeclareMathOperator{\Isom}{{\mathrm Isom}}

\DeclareMathOperator{\stab}{{\mathrm stab}}

\DeclareMathOperator{\End}{{\mathrm End}}
\DeclareMathOperator{\Aut}{{\mathrm Aut}}
\DeclareMathOperator{\Hom}{{\mathrm Hom}}
\DeclareMathOperator{\Ad}{{\mathrm Ad}}
\DeclareMathOperator{\ad}{{\mathrm ad}}
\DeclareMathOperator{\vol}{{\mathrm vol}}
\DeclareMathOperator{\SL}{{\mathrm SL}}

\def\g{\mathfrak g}

\DeclareMathOperator{\SO}{{\mathrm SO}}

\def\Ga{\Gamma}
\def\ga{\gamma}

\def\H{\mathbb H}

\def\R{\mathbb R}

\def\HH{\mathcal H}

\title{Cusps, Kleinian groups and Eisenstein series}

\author{Beibei Liu}
\address{Department of Mathematics, The Ohio State University, Columbus, OH 43210, USA}
\email{bbliumath@gmail.com}

\author{Shi Wang}
\address{Institute of Mathematical Sciences, ShanghaiTech University, Shanghai, China}
\email{shiwang.math@gmail.com}

\subjclass[2010]{}

\providecommand{\keywords}[1]{\textbf{\textit{Index terms---}} #1}

\date{}

\begin{document}

\begin{abstract}
We study the Eisenstein series associated to the full rank cusps in a complete hyperbolic manifold. We show that given a Kleinian group $\Ga<\Isom^+(\H^{n+1})$, each full rank cusp corresponds to a cohomology class in $H^{n}(\Ga, V)$ where $V$ is either the trivial coefficient or the adjoint representation. Moreover, by computing the intertwining operator, we show that different cusps give rise to linearly independent classes.

\end{abstract}


\maketitle

\section{Introduction}\label{sec:introduction} We say $\Ga$ is a Kleinian group if it is a discrete isometry subgroup of $G=\Isom^+(\H^{n})$, the orientation preserving isometry group of $\H^{n}$. One of the main themes in hyperbolic geometry is to study the number of cusps in the associated quotient manifold $\Ga\backslash \H^{n}$. When $\Ga$ contains parabolic elements, every cusp corresponds to a $\Ga$-conjugacy class of maximal parabolic subgroups in $\Ga$. In dimension $3$, the celebrated work of Sullivan \cite{Sullivan} shows that finitely generated Kleinian groups always have finitely many cusps, and the number of cusps is bounded by $5N-4$ where $N$ denotes the number of generators in $\Ga$. (See also the work of Kra \cite{Kra}). However, starting in dimension $4$, cusp finiteness theorem fails. The first example was due to Kapovich \cite{Kapovich1}, where he constructed a finitely generated free Kleinian group  $\Ga<\Isom^+(\H^{4})$ that has infinitely many rank one cusps. In a recent paper \cite{IMM}, Italiano, Martelli and Migliorini constructed a finitely generated Kleinian group $\Ga<\Isom^+(\H^{n})$ which has infinitely many full rank cusps in dimensions $5\leq n \leq 8$. Moreover, $\Ga$ can be made finitely presented in dimension $7$ and $8$. On the other hand, it is proved in  \cite{LW20} that the number of cusps is bounded by the first Betti number provided the critical exponent is smaller than $1$.

One general approach to show a cusp finiteness theorem is to first associate each maximal parabolic subgroup $\Ga_i<\Ga$ with a cohomology class $\alpha_i\in H^*(\Ga, V)$ after choosing a suitable coefficient module $V$, then to show the corresponding cohomology classes for different cusps are linearly independent. Finally, if we know the overall dimension of $H^*(\Ga, V)$ is finite, then the number of cusps must also be finite. For example, in Sullivan's proof he chose $V$ to be the polynomial space of degree at most 4, and constructed a cross homomorphism from $\Ga$ to $V$ (thus representing a class in $H^1(\Ga,V)$) via the Borel series associated to each cusp. Then he showed these representing classes are linearly independent, hence, the number of cusps is bounded in terms of  the first Betti number (number of generators of $\Ga$). In analogous to the Borel series, we can use the Eisenstein series to associate a cusp with a cohomology class in $H^{\ast}(\Ga, V)$. Much of work was done by Harder \cite{Harder, Ha1}, Schwermer \cite{Schwermer, Schwermerbook} and many others when $\Ga$ is an arithmetic lattice in a semisimple Lie group. Using the Borel-Serre compactification, the Eisenstein cohomology naturally arises from the cohomology of the boundary, which has deep relations to the arithmetic aspects of $\Ga$ such as the special values of $L$-functions.

The main purpose of this note is to extend the Eisenstein construction to the context of general Kleinian groups with full rank cusps. We make use of the Poincar\'e series to obtain absolute convergence of the Eisenstein series. Thus, each full rank cusp corresponds to a cohomology class on the quotient manifold. In order to distinguish these cohomology classes arising from different cusps, we compute the intertwining operators and use them to show that these cohomology classes are indeed linearly independent. The computation of the intertwining operators is very difficult in general. We follow the general approach of Harish-Chandra \cite{HarishChandra} but instead use the Lie group decompositions over the reals including the Bruhat and Langlands decompositions. In particular, our proof does not rely on the finite volume property or arithmeticity of $\Ga$. In the case of trivial coefficient, we prove,

\begin{theorem}\label{thm:trivial}
	Let $\Ga<\Isom^+(\mathbb H^{n+1})$ be a torsion free discrete subgroup. If  the critical exponent $\delta(\Ga)<n$ or $\Ga$ is of convergence type, then for any parabolic subgroup $\Ga_i<\Ga$ of rank $n$, and any generating cohomology class $\alpha_i\in H^{n}(\Ga_i,\R)\cong \R$, there is a harmonic form $E(\alpha_i)$ on $\Ga\backslash \H^{n+1}$ constructed via the Eisenstein series such that
	\begin{enumerate}
		\item The restriction homomorphism $H^{n}(\Ga,\R)\rightarrow H^n(\Ga_i, \R)$ sends $[E(\alpha_i)]$ to $\alpha_i$.
		\item If $\Ga_i, \Ga_j$ are not $\Ga$-conjugate, then the restriction homomorphism $H^{n}(\Ga,\R)\rightarrow H^n(\Ga_j, \R)$ sends $[E(\alpha_i)]$ to $0$.
	\end{enumerate} 
\end{theorem}


In the above theorem, the  \emph{critical exponent}    $\delta(\Ga)$ of $\Ga$ is defined as
\[\delta(\Gamma)=\inf\{s\;:\; \sum_{\ga\in \Ga}e^{-sd(O,\ga O)}<\infty\}.\]
Note that if $\Ga\subset \Isom(\H^{n+1})$, then $0\leq \delta(\Ga)\leq n$. For simplicity, we sometimes write $\delta$ for $\delta(\Ga)$ if the context is clear. The group $\Ga$ is said to be of convergence type if the above infimum is achieved. The additional assumption on the critical exponent or on the convergence type of $\Ga$ is to assure the absolute convergence of the Eisenstein series. This is necessary for our theorem to hold because in the case $\Ga$ is a non-uniform lattice (where $\delta=n$ and $\Ga$ is of divergence type), the degree $n$-homology classes coming from the cusps form a linearly dependent system, thus by Stokes' theorem, the result in our Theorem \ref{thm:trivial} will never hold. To our surprise, by examining the entire argument in our proof, the non-convergence of the Eisenstein series is the only place where it fails. However, if we choose the coefficient module to be the Lie algebra $\mathfrak g$ of $G$, equipped with the natural adjoint action of $\Ga$ inherited from $G$, then the absolute convergence issue will be resolved. This does not contradict to the example of non-uniform lattices since we do not have Stokes' theorem for $\mathfrak g$ coefficient. More precisely we prove,

\begin{theorem}\label{thm:Ad}
	Let $\Ga<\Isom^+(\mathbb H^{n+1})$ be a torsion free discrete subgroup. Then for any parabolic subgroup $\Ga_i<\Ga$ of rank $n$, and any cohomology class $\alpha_i\in H^{n}(\Ga_i,\Ad)$, there is a closed differential form $E(\alpha_i)$ on $\Ga\backslash \H^{n+1}$ constructed via the Eisenstein series such that
	\begin{enumerate}
		\item The restriction homomorphism $H^{n}(\Ga,\Ad)\rightarrow H^n(\Ga_i, \Ad)$ sends $[E(\alpha_i)]$ to $\alpha_i$. In particular, there is a surjective homomorphism
		\[H^{n}(\Ga,\Ad)\rightarrow H^n(\Ga_i, \Ad).\]	
		\item If $\Ga_i, \Ga_j$ are not $\Ga$-conjugate, then the restriction homomorphism $H^{n}(\Ga,\Ad)\rightarrow H^n(\Ga_j, \Ad)$ sends $[E(\alpha_i)]$ to $0$.
	\end{enumerate} 
\end{theorem}


Both Theorem $\ref{thm:trivial}$ and $\ref{thm:Ad}$ give a way to control the number of full rank cusps $N$ on the quotient manifold. The case of trivial coefficient implies $N\leq \beta_n(\Ga)$ (given that $\delta<n$ or that $\Ga$ is of convergence type) where $\beta_n(\Ga)$ denotes the $n$-th betti number of $\Ga$, but this is clear since a full rank cusp is always a topological end and $H^{n+1}(\Ga,\Ga_i,\R)=0$ unless $\Ga$ is a cocompact lattice. So the surjectivity of the restriction homomorphism $H^{n}(\Ga,\R)\rightarrow H^n(\Ga_i,\R)$ follows immediately from the long exact sequence for the pair $(\Ga,\Ga_i)$. In the case of adjoint representation, we obtain a similar bound.



\begin{corollary}\label{cor:cusp}
	Let $\Ga<\Isom^+(\mathbb H^{n+1})$ be a torsion free discrete subgroup. Then the number of full rank toric cusps of $\Ga\backslash \H^{n+1}$ is bounded by
	\[N\leq \frac{1}{n}\dim(H^n(\Ga,\Ad)).\]
\end{corollary}

\begin{remark}
The reason why we need to add the toric cusp condition is that in general $H^{n}(\Ga_i, \Ad)$ could be trivial (See Proposition \ref{prop:nontoric} and Remark \ref{rem:non-toric}). If the group $\Ga$ is LERF, then we can always pass onto a finite cover of $\Ga\backslash \H^{n+1}$ to assure a given full rank cusp is toric. 
\end{remark}


\subsection*{Organization of the paper} In Section \ref{sec:pre}, we review vector-valued differential forms, their equivalent perspective as functions on Lie groups, and the Lie algebra cohomology. In Section \ref{sec:nform}, we construct a cohomology class for each full rank parabolic fixed point with the coefficient either $\R$ or the Lie algebra of $\Isom^+(\H^{n+1})$ with the adjoint representation. In Section \ref{sec:eisen}, we construct the Eisenstein series and discuss its closeness and convergence. In Section \ref{sec:intertwining}, we investigate the restriction of the Eisenstein series to the horosphere corresponding to any given cusp. In Section \ref{sec:proof}, we prove Theorem \ref{thm:trivial}, \ref{thm:Ad} and Corollary \ref{cor:cusp}.  

\subsection*{Acknowledgments} We would like to thank Dubi Kelmer and Misha Kapovich for the helpful emails, and Grigori Avramidi for the helpful discussions. We are also grateful to the referees for careful reading and comments, and Max Planck Institute for Mathematics in Bonn for its hospitality and financial support where this work was initiated. The second author also thanks Michigan State University for its hospitality where most of the work was done. The first author is partially supported by the NSF grant DMS-2203237.



\section{Preliminary}
\label{sec:pre}

\subsection{Vector-valued differential forms}\label{sec:differential forms}
Let $V$ be a finite dimensional real vector space and $\rho: G\rightarrow \Aut(V)$ be any continuous (and hence smooth) representation. Since $\rho$ restricts to any discrete subgroup $\Ga$, and $\Ga$ naturally acts from the left on $X=G/K$ where $K$ is the stabilizer of a point in $G$, it follows that $\Ga$ also left acts on the trivial bundle $V\times X$ via
\[\gamma \cdot (v,x)=(\rho(\gamma)(v),\gamma x).\]
Endowed with the trivial connection on $V\times X$, it induces a flat bundle structure on the quotient manifold $M=\Ga\backslash X$, which we denote by $V_{\rho}$. It is known that the cohomology with the associated local system $H^*(\Ga,V_{\rho})$ can be computed using the DeRham complex $\Omega^*(X,V)^\Ga$, where the codifferential operator $d:\Omega^k(X,V)\rightarrow \Omega^{k+1}(X,V)$ is defined by
\begin{align*}\label{eq:codifferential}
	d\omega(X_1,..,X_{k+1})&:=\sum_i (-1)^{i+1} X_i\omega(X_1,...,\widehat X_i,...,X_{k+1})\\
	&+\sum_{i<j}(-1)^{i+j}\omega([X_i,X_j],X_1,...,\widehat X_i,..., \widehat X_j, ..., X_{k+1}).
\end{align*}
To this end, any cohomology class in $H^*(\Ga,V_\rho)$ can be represented by a $\Ga$-invariant $V$-valued closed differential form on $X$.

\subsection{Matsushima-Murakami formalism}

For the purpose of computations, it is convenient to view alternatively the above mentioned vector valued differential forms as a smooth function from $G$ to $\Hom(\Lambda^*\mathfrak g, V)$. We follow the original treatment in \cite[Section 4]{Matsushima-Murakami}.

Given any $\eta\in \Omega^*(X,V)^\Ga$, we can pull back the differential form to $G$ under the projection $\pi:G\rightarrow X=G/K$,  followed by a twist of a $G$-action. Define a differential form $\widetilde{\eta}\in\Omega^*(G, V)$ by
\[\widetilde\eta_s:=\rho(s^{-1})(\eta\circ \pi)_s,\quad \forall s\in G.\]
Then one can check $\eta\in \Omega^*(X,V)^\Ga$ if and only if $\widetilde{\eta}$ satisfies
\begin{enumerate}
	\item $\widetilde{\eta}\circ L_\ga=\widetilde{\eta}$ for any $\gamma\in \Ga$,
	\item $\widetilde{\eta}\circ R_k=\rho(k^{-1})\widetilde{\eta}$ for any $k\in K$, and
	\item $i(Y)\widetilde{\eta}=0$ for any $Y\in \mathfrak k$,
\end{enumerate}
where $\mathfrak{k}$ denotes the Lie algebra of $K$, and $L_{\ga}, R_{k}$ denote the left and right multiplications. 
Furthermore, we can view $\widetilde{\eta}$ as a function on $G$ whose values are in $\Hom(\Lambda^*\mathfrak g, V)$ by identifying $T_sG$ with $T_eG\cong \mathfrak g$ via left translation, and the above constraints then turn into
\begin{enumerate}
	\item $\widetilde{\eta}(\gamma g)=\widetilde{\eta}(g)$ for any $\gamma\in \Ga$ and $g\in G$,
	\item $\widetilde{\eta}(gk)=\Ad_{\mathfrak g}^*(k^{-1})\otimes\rho(k^{-1})(\widetilde{\eta}(g))$ for any $k\in K$ and $g\in G$, and
	\item $i(Y)\widetilde{\eta}=0$ for any $Y\in \mathfrak k$.
\end{enumerate}
where $\Ad_{\mathfrak g}^*$ is the dual adjoint representation of $G$ on $\Lambda^*(\mathfrak g)$.  Fix a basepoint on $X$, we write $\mathfrak{g}=\mathfrak{k}\oplus \mathfrak{p}$ the Cartan decomposition. Then there is a natural identification between $\Omega^*(X,V)^\Ga$ and functions $\varphi\in C^\infty(G/\Ga,\Hom(\Lambda^*\mathfrak p,V))$ which satisfies
\[\varphi(gk)=\Ad_{\mathfrak p}^*(k^{-1})\otimes\rho(k^{-1})(\varphi(g)),\quad \forall k\in K, g\in G.\]

Under such identification, the coboundary operator $d: \Omega^k(X,V)\rightarrow \Omega^{k+1}(X,V)$ as described gives rise to the coboundary operator \[d:C^\infty(G,\Hom(\Lambda^k\mathfrak g,V))\rightarrow C^\infty(G,\Hom(\Lambda^{k+1}\mathfrak g,V))\]
given by (\cite[Proof of Proposition 4.1]{Matsushima-Murakami})
\begin{equation}\label{eq:coboundary-operator}
	\begin{aligned}
		d\varphi(X_1,...,X_{k+1})&=\sum_{i} (-1)^{i+1} (X_i+\rho(X_i))\varphi(X_1,...,\widehat X_i,...,X_{k+1})\\
		&+\sum_{i<j}(-1)^{i+j}\varphi([X_i,X_j],X_1,...,\widehat X_i,..., \widehat X_j, ..., X_{k+1}),
	\end{aligned}
\end{equation}

Here we abuse notation and still use $\rho$ to denote the induced Lie algebra representation $\rho: \mathfrak{g}\rightarrow \End(V)$. Note that in the original statement of \cite[Proposition 4.1]{Matsushima-Murakami}, the second term of the above equation \eqref{eq:coboundary-operator} vanishes. This is because their function $\varphi$ is valued on $\Hom(\Lambda^*\mathfrak p,V)$ and that $[X_i,X_j]\in \mathfrak k$ for any pair $X_i, X_j\in \mathfrak p$. However, for the purpose of computation, besides the usual coset model $\H^{n}=G/K=\SO^+(n,1)/\SO(n)$, we will also use $\H^{n}=P_\xi/K_\xi $, where $P_{\xi}$ is a maximal parabolic group for the parabolic fixed point $\xi$, and $K_{\xi}=P_{\xi}\cap K$. Let us be a little more verbose here as this description is essential to the computations of $\phi_{\xi,s}$ in Section \ref{sec:nform}, \ref{sec:intertwining} and \ref{sec:proof}.

Under the Langlands decomposition, we have $P_\xi=N_\xi A_\xi K_\xi$, and accordingly the Lie algebra splits as $\mathfrak{p}_\xi= \mathfrak n_\xi\oplus \mathfrak a_\xi\oplus \mathfrak m_\xi $. Thus from the above discussions, any differential form $\eta\in \Omega^k(X,V)$ can be viewed as a function $\widetilde{\eta}\in C^\infty(P_\xi, \Hom(\Lambda^{k}\mathfrak p_\xi, V))$ which satisfies

\begin{enumerate}
	\item $\widetilde{\eta}(pm)=\Ad_{\mathfrak p_\xi}^*(m^{-1})\otimes\rho(m^{-1})(\widetilde{\eta}(p))$ for any $m\in K_\xi$ and $p\in P_\xi$, and
	\item $i(Y)\widetilde{\eta}=0$ for any $Y\in \mathfrak m_\xi$,
\end{enumerate}
where the second property shows that we can further view $\widetilde{\eta}$ as in $C^\infty(P_\xi, \Hom(\Lambda^{k}\mathfrak (\mathfrak a_\xi\oplus \mathfrak n_\xi), V))$. If $H<P_\xi$, then $\eta$ is $H$-invariant if and only if $\widetilde \eta$ is $H$-left invariant as a function. Note that the Lie bracket $[\mathfrak a_\xi ,\mathfrak n_\xi]$ stays in $\mathfrak n_\xi$, so in particular the second term in \eqref{eq:coboundary-operator} will possibly be nonzero. (See Proposition \ref{prop:adclosed} and compare the proof of \cite[Lemma 3.1]{Harder}.) 


\begin{remark}
	Our convention uses left action of $G$ on $X=G/K$, which is different from that in \cite{Harder}. So there are sign differences in the expression of the coboundary operators.
\end{remark}


\begin{lemma}\label{lem:commute}
	For any $g\in G$, if $L_g$ denotes the left action on $C^\infty(G,\Hom(\Lambda^*\mathfrak g,V))$, i.e. $(L_g\varphi)|_a=\varphi|_{ga}$ for any $\varphi\in C^\infty(G,\Hom(\Lambda^{k}\mathfrak g,V))$ and any $a\in G$, then
	\[L_g\circ d=d\circ L_g.\]
	In particular, $\varphi$ is closed if and only if $L_g\varphi$ is closed, and $\varphi$ is a coboundary if and only if $L_g\varphi$ is a coboundary.
\end{lemma}

\begin{proof}
	For any $a\in G$, $\varphi\in C^\infty(G,\Hom(\Lambda^k\mathfrak g,V))$, and any $X_1,..., X_{k+1}\in \mathfrak g$. We do the following direct computations:
\begin{align*}
	d(L_g \varphi)|_a(X_1,...,X_{k+1})&=\sum_{i} (-1)^{i+1} (X_i+\rho(X_i))(L_g\varphi)|_a(X_1,...,\widehat X_i,...,X_{k+1})\\
	&\quad\quad+\sum_{i<j}(-1)^{i+j}(L_g\varphi)|_a([X_i,X_j],X_1,...,\widehat X_i,..., \widehat X_j, ..., X_{k+1})\\
	&=\sum_{i} (-1)^{i+1} (X_i+\rho(X_i))\varphi|_{ga}(X_1,...,\widehat {X_i},...,X_{k+1})\\
	&\quad\quad+\sum_{i<j}(-1)^{i+j} \varphi|_{ga}([X_i, X_j],X_1,...,\widehat {X_i},..., \widehat {X_j}, ..., X_{k+1})\\
	&=d\varphi|_{ga}(X_1,...,X_{k+1})\\
	&=L_g(d\varphi)|_a(X_1,...,X_{k+1}).
\end{align*}
Thus $L_g\circ d=d\circ L_g$.
\end{proof}

\subsection{Lie algebra cohomology} Let $\mathfrak g$ be a Lie algebra and $\rho:\mathfrak g\rightarrow \End(V)$ be a Lie algebra representation. We define the Chevalley–Eilenberg complex by \[\dots\rightarrow \Hom(\Lambda^k \mathfrak g, V)\xrightarrow{d} \Hom(\Lambda^{k+1} \mathfrak g, V)\rightarrow \dots\]
and the coboundary operator is given by
\begin{align*}
	d\varphi(X_1,...,X_{k+1})&=\sum_{i} (-1)^{i+1} \rho(X_i)\varphi(X_1,...,\widehat X_i,...,X_{k+1})\\
	&+\sum_{i<j}(-1)^{i+j}\varphi([X_i,X_j],X_1,...,\widehat X_i,..., \widehat X_j, ..., X_{k+1}).
\end{align*}
The cohomology induced by the above cochain complex is called the Lie algebra cohomology with $V$-coefficient, denoted by $H^*(\mathfrak g, V_\rho)$.

Since we only work with specific Lie algebras and representations, we will make simplifications by setting $G=\Isom^+(\mathbb H^{n+1})\cong \SO^+(n+1,1)$, and setting either $V=\mathfrak g$ and $\rho:G\rightarrow \End (\mathfrak g)$ the adjoint representation, or $V=\mathbb R$ and $\rho$ the trivial representation. Let $U<G$ be a maximal unipotent subgroup associated to some chosen maximal abelian subgroup $A<G$ such that $U$ is expanding. In the case of adjoint representation, we denote $\mathfrak u, \mathfrak g$ the Lie algebra of $U, G$, and $\rho: \mathfrak u\rightarrow \End(\mathfrak g)$  the restriction of the adjoint representation $\rho$. Diagonalized by the adjoint action of $A$, the vector space $V$ (under the restricted root space decomposition) decomposes as $V=V_{-2}\oplus V_0\oplus V_2$ and that the Lie algebra $\mathfrak u= V_2$. 

\begin{lemma}\label{lem:lie alg cohomology}
	Following the same notations above, if $\{u_1,...,u_n\}$ is a basis of $\mathfrak u$, and $v\in V_{-2}$, then there is a natural isomorphism
	\[J:\;V_{-2}\cong H^{n}(\mathfrak u, V_{\rho}),\]
	given by
	\[v\mapsto (u_1^*\wedge...\wedge u_n^*)\otimes v,\]
    where $\{u_1^*,...,u_n^*\}$ represents the dual basis on $\mathfrak u^*$.
\end{lemma} 

\begin{proof}
	Since $n=\dim \mathfrak u$, the cochain complex stops in dimension $n$, so $(u_1^*\wedge...\wedge u_n^*)\otimes v$ is automatically closed. For the injectivity of $J$, it suffices to show that for any nonzero $v\in V_{-2}$, the closed form $(u_1^*\wedge...\wedge u_n^*)\otimes v$ does not lie in the image of
	\[d: \Hom(\Lambda^{n-1}\mathfrak u, V)\rightarrow \Hom(\Lambda^{n}\mathfrak u, V).\]
	We write an arbitrary element in $\Hom(\Lambda^{n-1}\mathfrak u, V)$ as
	\[\varphi= \sum_{i=1}^n \left(u_1^*\wedge...\wedge \widehat{u_i^*}\wedge...\wedge u_n^*\right) \otimes A_i,\]
	for some $A_i\in V$. Then we compute
	\[d\varphi= \left(u_1^*\wedge...\wedge u_n^*\right)\otimes\left(\sum_{i=1}^n (-1)^{i+1}\rho(u_i)A_i\right). \]
	Since $[V,V_2]=V_0\oplus V_2$, we see that $\left(\sum_{i=1}^k (-1)^{i+1} \rho(u_i)A_i\right)\in V_0\oplus V_2$, and in particular it does not lie in $V_{-2}$. Thus $(u_1^*\wedge...\wedge u_n^*)\otimes v$ is not a coboundary, and it represents a nontrivial cohomology class. This shows the map $J$ is injective.
	
	To show $J$ is surjective, we need to show any element $(u_1^*\wedge...\wedge u_{n}^*)\otimes v$ where $v\in V_0\oplus V_2$ is a coboundary. Since $[V,V_2]=V_0\oplus V_2$, we can write $v=\sum_{i=1}^{n}[u_i,v_i]$ for some $v_i\in V$. Now if we set
	\[\varphi= \sum_{i=1}^n \left(u_1^*\wedge...\wedge \widehat{u_i^*}\wedge...\wedge u_n^*\right) \otimes (-1)^{i+1}v_i\in \Hom(\Lambda^{n-1}\mathfrak u, V),\]
	then $d\varphi= (u_1^*\wedge...\wedge u_{n}^*)\otimes v$. This proves that $(u_1^*\wedge...\wedge u_{n}^*)\otimes v$ is a coboundary. Thus $J$ is surjective.
\end{proof}

One important aspect of the Lie algebra cohomology is that it sometimes relates to the group cohomology of a Lie group $G$, and those classes can be identified with certain $G$-invariant differential forms which are harmonic. We will use the following Van-Est isomorphism theorem in the context of abelian Lie groups.

\begin{theorem}\label{thm:vanEst} \cite{VanEst}
	Let $U$ be an $n$-dimensional Lie group isomorphic to $\mathbb R^n$, and $Z<U$ be a torsion-free cocompact lattice. Let $\rho: U\rightarrow \Aut(V)$ be a representation which induces the Lie algebra representation $\rho':\mathfrak u\rightarrow \End(V)$. Then there is an natural isomorphism 
	\[\Phi:H^*(Z,V_\rho)\cong H^*(\mathfrak u, V_{\rho'}),\]
	explicitly given by the following, 
	for any $Z$-invariant closed differential form $\omega\in \Omega^*(U,V_\rho)$, set
	\[\Phi(\omega)=\int_{U/Z}\omega(x)d\mu(x),\]
	where $d\mu$ is the Haar measure on $U/Z$. Thus, $\Phi(\omega)$ is an $U$-invariant differential form that can be identified with an element in $\Hom(\Lambda^*(\mathfrak u), V_{\rho'})$.
\end{theorem}

\section{Construction of cohomology classes from a cusp}
\label{sec:nform}

Suppose that  $G=\Isom^+(\mathbb H^{n+1})\cong \SO^+(n+1,1)$, and $\Ga<G$ is a torsion-free discrete subgroup. For each cusp on the quotient manifold $\Ga\backslash \mathbb H^{n+1}$, when lifted to the universal cover, it associates to a $\Ga$-orbit  $\Ga \xi$ on $\partial _\infty \mathbb H^{n+1}$, for some $\xi\in \partial _\infty \mathbb H^{n+1}$. The fundamental group of the cusp is isomorphic to $\Ga_\xi=\Ga\cap P_\xi$ where $P_{\xi}=\stab_G(\xi)<G$ is the real parabolic subgroup at $\xi$. 

Fix a base point $O\in \H^{n+1}$, and let $\mathcal H_\xi(1)$ be the horosphere of the parabolic fixed point $\xi$ through $O$. Under the induced Riemannian metric, it is isometric to the standard Euclidean space $\mathbb R^{n}$. It is known that $\Ga_\xi$ preserves and acts isometrically on $\mathcal H_\xi(1)$ and in fact by Bieberbach's theorem it acts cocompactly on a $k$-dimensional Euclidean subspace $\mathbb E_\xi^k$ of $\mathbb R^{n}\cong \mathcal H_\xi(1)$. It follows that $\Ga_\xi$ has a finite index abelian normal subgroup $Z_\xi$ which acts on $\mathbb E_\xi^k$ by translations. We call $k$ the \emph{rank} of the cusp at $\xi$. For the purpose of this paper, we will only consider the full rank case $k=n$, and from now on all cusps are assumed to be full rank.

Under the Langlands decomposition, the real parabolic subgroup decomposes as $P_\xi=N_\xi A_\xi K_\xi$, where $A_\xi$ is the maximal abelian subgroup which acts by translation on the geodesic connecting $O$ and $\xi$. Geometrically, $K_{\xi}\cong \SO(n)$ and $N_{\xi}\cong \mathbb R^{n}$ are the rotations and translations on  $\mathcal H_\xi(1)$ respectively. Denote $K_{\xi}N_{\xi}=N_\xi K_\xi$ by  $P_{\xi}(1)$.  Then it is indeed the orientation preserving isometry group of $\mathcal H_\xi(1)$, so $\Ga_\xi$ is a discrete subgroup in $P_\xi(1)$. We say the cusp is \emph{toric} if $\Ga_\xi$ is isomorphic to $\mathbb Z^n$ since under the quotient the cusp is homeomorphic to $\mathbb T^n\times [0,\infty)$.

For each cusp, we lift to the universal cover and choose an arbitrary parabolic fixed point $\xi\in \partial_{\infty}\H^{n+1}$ representing the cusp. For the convenience, we will describe the following construction under the upper-half plane model $\mathbb H^{n+1}=\{(y,x_1,...,x_{n})\in \mathbb R^+\times \mathbb R^{n}\}$. We assume $O=(1,0,...,0)$ and $\xi$ is in the positive infinity of the $y$-axis, thus $\mathcal H_\xi(1)=\{(y,x_1,...,x_n): y=1\}$. For a different $\xi'\in  \partial_{\infty}\H^{n+1}$, the construction differs by a $k(\xi,\xi')$-conjugate where $k(\xi,\xi')\in K$ is any element which sends $\xi$ to $\xi'$.

\subsection{Trivial coefficients}
In the case of trivial representation $V=\R$, we define $\phi_\xi$ to be the canonical volume form on $\mathcal H_\xi(1)$, that is,
\[(\phi_\xi)_x=dx_1\wedge...\wedge dx_n\]
for any $x\in \HH_\xi(1)$. It is convenient (for the purpose of computation) to view $\phi_\xi$ also as a function in $C^\infty\left(P_\xi(1), \Lambda^n\mathfrak n_\xi^*\right)$ according to the discussion in Section \ref{sec:differential forms}. We choose the orthonormal frame $\{u_1,...,u_n\}$ on $\mathfrak n_\xi$ normalized so that each $u_i$ exponentiates to the unit translation on $\HH_\xi(1)\cong \R^{n}$ along the $x_i$-axis. Thus $\phi_\xi$ satisfies,
\begin{enumerate}
	\item $\phi_\xi(n)=\phi_\xi(1)=(u_1^*\wedge...\wedge u_n^*)$, for all $n\in N_\xi$,
	\item $\phi_\xi(pm)=\Ad_{\mathfrak n_\xi}^*(m^{-1})(\phi_\xi(p))$ for all $m\in K_\xi$ and $p\in P_\xi(1)$,
\end{enumerate}
where $u_i^*\in \mathfrak n_\xi$ is the dual vector of $u_i$. We wish to extend the differential form to the entire $\H^{n+1}$, or equivalently, extend $\phi_\xi\in C^\infty\left(P_\xi(1), \Lambda^n\mathfrak n_\xi^*\right)$ to a function in $C^\infty\left(P_\xi, \Lambda^n(\mathfrak a_\xi^* \oplus \mathfrak n_\xi^*)\right)$. Following \cite{Harder}, we introduce the following \emph{degree $s$} extension $\phi_{\xi,s}$ of $\phi_\xi$. Let $t_\xi: A_\xi\rightarrow \R$ be the character that corresponds to the positive root on $\mathfrak a_\xi$, that is, for any $a\in A_\xi$ and $v\in \mathfrak g$, we have
	\begin{equation}\label{eq:character}
	\Ad(a)v= \begin{cases}
		t_{\xi}^2(a)v &\text{if } v\in \mathfrak g_2\\
		0 &\text{if } v\in \mathfrak g_0\\
		t_{\xi}^{-2}(a)v&\text{if } v\in \mathfrak g_{-2}
	\end{cases},
\end{equation}	
where $\mathfrak g=\mathfrak g_2\oplus \mathfrak g_0\oplus\mathfrak g_{-2}$ is the root space decomposition corresponding to $A_\xi$ (take $\xi$ to be the positive direction). Therefore, if $T\in \mathfrak a_\xi$ is the vector such that $[T,u]=2u$ for all $u\in \mathfrak n_\xi$, then we have $(dt_\xi)_a(T)=t_\xi(a)$, or that $dt_\xi/t_\xi$ is dual to $T$. We also write $t_{\xi, a}$ for $t_\xi(a)$ when convenient, and omit the subscript $\xi$ when there is no confusion in the context. Now we define the degree $s$ extension $\phi_{\xi, s}$ of $\phi_\xi$ by

\begin{enumerate}
	\item $\phi_{\xi,s}(n)=\phi_\xi(1)=(u_1^*\wedge...\wedge u_n^*)$, for all $n\in N_\xi$,
	\item $\phi_{\xi,s}(p)=\Ad_{\mathfrak n_\xi}^*(m^{-1})(\phi_\xi(1))t_{\xi,a}^s$, for all $p=nam\in N_\xi A_\xi K_\xi$ under the unique Langlands decomposition.
\end{enumerate}
Thus in view of Section \ref{sec:differential forms}, $\phi_{\xi, s}\in C^\infty\left(P_\xi, \Lambda^n(\mathfrak a_\xi^* \oplus \mathfrak n_\xi^*)\right)$ represents an $N_\xi$-invariant differential form on $\H^{n+1}$.

\begin{prop}
\label{trivialclosed}
	In case of trivial representation, $\phi_{\xi,s}$ is closed when $s=2n$.
\end{prop}

The proof is deferred in Proposition \ref{prop:adclosed}.

\begin{remark}\label{rem:harmonic}
	It is much easier to find this unique closed $N_\xi$-invariant differential form $\phi_{\xi,s}$ when it is viewed as in $\Omega^n(\H^{n+1},\R)^{N_\xi}$, which is simply the pull back form of $\phi_\xi$ under the natural projection $\H^{n+1}\rightarrow \HH_\xi(1)$. In other words, it is the harmonic form
	\[(\phi_{\xi,s})_x=dx_1\wedge...\wedge dx_n\]
	for any $x\in \H^{n+1}$. However, when it comes to the computation of intertwining operators, Treating $\phi_{\xi,s}$ as a function will be more convenient for us.
\end{remark}
\subsection{Adjoint representation} Now we assume $V=\mathfrak g$ and $\rho$ is the adjoint representation. We first define a harmonic form on $\mathcal H_\xi(1)$. By Lemma \ref{lem:lie alg cohomology}, in order to obtain non-trivial cohomology classes, we need to take $V_{-2}$ sections. Let $v_\xi\in V_{-2}$ be an arbitrary non-zero vector. We construct the unique (left) $N_\xi$-invariant $n$-form $\phi_\xi$ on $\mathcal H_\xi(1)$ by setting the initial value
\[(\phi_\xi)_O:=(dx_1\wedge...\wedge dx_n)\otimes v_\xi.\]
Since $v_\xi$ does not commute with $\mathfrak n_\xi$, the $N_\xi$-action will twist the fibre as we vary the point on the horosphere. More precisely, the differential $n$-form $\phi_\xi$ can be globally defined as
\[(\phi_\xi)_{uO}:=(dx_1\wedge...\wedge dx_n)\otimes \rho(u)(v_\xi) ,\quad \forall u\in N_\xi.\]
Alternatively, we can view $\phi_\xi$ as a function in $C^\infty\left(P_\xi(1), \Hom(\Lambda^n\mathfrak n_\xi, V)\right)$ which satisfies,
\begin{enumerate}
	\item $\phi_\xi(n)=\phi_\xi(1)=(u_1^*\wedge...\wedge u_n^*)\otimes v_\xi $, for all $n\in N_\xi$,
	\item $\phi_\xi(pm)=\Ad_{\mathfrak n_\xi}^*(m^{-1})\otimes\rho(m^{-1})(\phi_\xi(p))$ for all $m\in K_\xi$ and $p\in P_\xi(1)$,
\end{enumerate}
where again $u_i\in \mathfrak u_\xi$ is the unique vector that exponentiates the unit translation in $x_i$, and $u_i^*$ is its dual vector. Now we define the degree $s$ extension $\phi_{\xi,s}$ of $\phi_\xi$ by
\begin{enumerate}
	\item $\phi_{\xi,s}(n)=\phi_\xi(1)=(u_1^*\wedge...\wedge u_n^*)\otimes v_\xi $, for all $n\in N_\xi$,
	\item $\phi_{\xi,s}(p)=\phi_{\xi,s}(nam)=\Ad_{\mathfrak n_\xi}^*(m^{-1})\otimes\rho(m^{-1})(\phi_\xi(1))t_{\xi,a}^s$, for all $p\in P_\xi$.
\end{enumerate}
Thus $\phi_{\xi,s}\in C^\infty\left(P_\xi, \Hom(\Lambda^n\mathfrak p_\xi, V)\right)$ represents an $N_\xi$-invariant $V$-valued differential form on $\H^{n+1}$, where $\mathfrak p_\xi$ is the Lie algebra of $P_\xi$. The following proposition computes the differential of $\phi_{\xi,s}$ in the most generality (for later use in Section \ref{sec:proof}), and in particular it characterizes when $\phi_{\xi,s}$ is closed. This is essentially proved in \cite[Lemma 3.1]{Harder}, but for completeness we add it here.

\begin{prop}\label{prop:adclosed}
	Let $\phi_{\xi,s}$ be the degree $s$ extension of $\phi_\xi$ where $\phi_\xi$ is defined as above except now we allow $v_\xi\in V_{\ell}$ for $\ell\in \{-2, 0, 2\}$, and $V$ is either with the trivial representation or with the adjoint representation. Then
	\[d\phi_{\xi, s}=\begin{cases}
		(s-2n+\ell)\cdot  \dfrac{dt_{\xi}}{t_{\xi}}\wedge \phi_{\xi, s} & \text{ if $\rho$ is adjoint representation}\\
		(s-2n) \cdot \dfrac{dt_{\xi}}{t_{\xi}}\wedge \phi_{\xi, s}&\text{ if $\rho$ is trivial representation}
	\end{cases}.
	\]
where $t_{\xi}: A_{\xi}\rightarrow \mathbb{R}$ is the character corresponds to the positive root on $\mathfrak{a}_{\xi}$ defined as in \eqref{eq:character}. 
\end{prop}

\begin{proof} Let $X_1,\cdots, X_{n+1}\in \mathfrak p_\xi= \mathfrak k_\xi\oplus \mathfrak a_\xi\oplus \mathfrak n_\xi$, we need to compute $d\phi_{\xi, s}(X_1,\cdots,X_{n+1})$. Recall that
\begin{align*}
		d\phi_{\xi, s}(X_1, \cdots ,X_{n+1})&=\sum_{i} (-1)^{i+1} (X_i+\rho(X_i))\phi_{\xi, s}(X_1,\cdots,\widehat X_i,\cdots,X_{n+1})\\
		&+\sum_{i<j}(-1)^{i+j}\phi_{\xi, s}([X_i,X_j],X_1,\cdots,\widehat X_i,\cdots, \widehat X_j, \cdots, X_{n+1}).
\end{align*}
The form $\phi_{\xi}$ is a closed $n$-form on the horosphere $\HH_{\xi}(1)$. It suffices to only consider the term $d\phi_{\xi, s}(T, X_{1}, \cdots, X_{n})$ where $T\in \mathfrak a_\xi$ and $X_{i}\in \mathfrak{n}_{\xi}$ since all the other terms vanish. (See \cite[(4.10)]{Matsushima-Murakami}.) We choose $X_i=u_i\in \mathfrak n_\xi$ and normalize $T$ such that $[T,X_i]=2X_i$, i.e. $T$ is dual to $dt_{\xi}/t_{\xi}$. Then at any point $p=na\in N_\xi A_\xi$, we have
\begin{align*}
	d\phi_{\xi, s}(T, X_{1},\cdots, X_{n})&= (T+\rho(T))\phi_{\xi, s}(X_1,\cdots, X_{n})\\
	&+\sum_{i}(-1)^{i}\phi_{\xi, s}([T,X_i],X_1,\cdots,\widehat X_i, \cdots, X_{n}).
\end{align*}
The second term can be simplified as
\begin{align*}
	\sum_{i}(-1)^{i}\phi_{\xi, s}([T,X_i],X_1,\cdots,\widehat X_i, \cdots, X_{n})&=-2\sum_{i=1}^{n}\phi_{\xi, s}(X_{1}, \cdots, X_{i}, \cdots, X_{n})\\
	&=-2n\cdot \phi_{\xi, s}(X_1, \cdots, X_n).
\end{align*}
Furthermore, we compute
$$\rho(T)\phi_{\xi, s}(X_{1}, \cdots, X_{n})=\begin{cases}
	\ell\cdot \phi_{\xi, s}(X_1, \cdots, X_n) & \text{ if $\rho$ is adjoint representation}\\
     0&\text{ if $\rho$ is trivial representation}
\end{cases}$$
and
$$T\phi_{\xi, s}(X_{1}, \cdots, X_{n})=s\cdot \phi_{\xi, s}(X_1, \cdots, X_n).$$
Therefore,
$$d\phi_{\xi, s}=\begin{cases}
	(s-2n+\ell)\cdot  \dfrac{dt_{\xi}}{t_{\xi}}\wedge \phi_{\xi, s} & \text{ if $\rho$ is adjoint representation}\\
	(s-2n) \cdot \dfrac{dt_{\xi}}{t_{\xi}}\wedge \phi_{\xi, s}&\text{ if $\rho$ is trivial representation}
\end{cases}.
$$
\end{proof}

\begin{cor}\label{cor:closed}
In case of adjoint representation and $v_\xi\in V_{-2}$, $\phi_{\xi,s}$ is closed when $s=2n+2$.	
\end{cor}

\subsection{Non-toric cusps}\label{sec:nontoric}
In the case of toric cusps, we know $\Ga_\xi<N_\xi$, hence the above constructed differential form $\phi_\xi$ (hence also its extension $\phi_{\xi,s}$) is automatically $\Ga_\xi$-invariant, and by the Van-Est isomorphism (Theorem \ref{thm:vanEst}) every cohomology class in $H^n(\Ga_\xi,\Ad)$ can be represented this way. For a general non-toric cusp, $\Ga_\xi$ is a Bieberbach group. With trivial coefficient, we still have $H^n(\Ga_\xi,\R)\cong H^n(Z_\xi,\R)\cong H^n(\mathfrak u,\R)\cong \R$, and the canonical volume form is $\Ga_\xi$-invariant. 

However, with adjoint representation it is less clear. Since $Z_\xi$ is a finite index normal subgroup of $\Ga_\xi$, the transfer map $\iota :H^k(Z_\xi,\Ad)\rightarrow H^k(\Ga_\xi,\Ad)$, given by taking the average over the finite group action of $Z_\xi\backslash\Ga_\xi$, is a left inverse of the restriction map $i^*$. That is, the following composition
\[H^k(\Ga_\xi,\Ad)\xrightarrow{i^*} H^k(Z_\xi,\Ad)\xrightarrow{\iota} H^k(\Ga_\xi,\Ad)\]
is the identity map. In particular, $i^*$ is injective and $\iota$ is surjective. Therefore, we can identify cohomology classes of $H^k(\Ga_\xi,\Ad)$ with their images in $H^k(Z_\xi,\Ad)$. 

\begin{prop}\label{prop:nontoric}
	For any cohomology class $\alpha\in H^n(\Ga_\xi,\Ad)$, there is a unique  $N_\xi$-invariant, $\Ga_\xi$-invariant closed differential form $\overline{\psi_\alpha}\in C^\infty\left(P_\xi(1), \Hom(\Lambda^n\mathfrak n_\xi, V)\right)$ representing $\alpha$ whose initial value satisfies,
	\[\overline{\psi_\alpha}(1)=(u_1^*\wedge\dots \wedge u_n^*)\otimes v\]
	for some $v\in V_{-2}$. Moreover, $\Theta v$ must be fixed by $\Ga_\xi$ under the adjoint action where $\Theta: \mathfrak{g}\rightarrow \mathfrak{g}$ is the Cartan involution associated with the base point $O$. Thus, $H^n(\Ga_\xi,\Ad)\neq0$ if and only if $\Ga_\xi$ fixes a nontrivial vector in $\mathfrak n_\xi$.
\end{prop}

\begin{proof}
	We first choose the harmonic representative of $i^*(\alpha)$ in $H^n(Z_\xi,\Ad)\cong H^n(\mathfrak n_\xi,\ad)$ under the Van-Est isomorphism, and we denote it by $\psi_\alpha$. Then by Lemma \ref{lem:lie alg cohomology}, we can choose an $N_\xi$-invariant form $\psi_\alpha\in C^\infty\left(P_\xi(1), \Hom(\Lambda^n\mathfrak n_\xi, V)\right)$ such that $\psi_\alpha(1)=(u_1^*\wedge\dots \wedge u_n^*)\otimes v_\alpha$ for some $v_\alpha\in V_{-2}$. Since $\psi_{\alpha}$ is a top form, it is closed. We denote $\overline{\psi_\alpha}$ the image of $\psi_\alpha$ under the transfer map, hence by definition we have
	\[\overline{\psi_\alpha}(p)=\frac{1}{D}\sum_{[\ga_i]\in Z_\xi\backslash \Ga_\xi}\psi_\alpha (\ga_i \cdot p),\]
	where $D=[\Ga_\xi: Z_\xi]$.
	By Lemma \ref{lem:commute}, $\overline{\psi_\alpha}$ is closed. Clearly it is $\Ga_\xi$-invariant,  and since $\iota\circ i^*=Id$, $\overline{\psi_\alpha}$ represents $\alpha$. To see it is also $N_\xi$-invariant we compute for any $u\in N_\xi$ and $p\in P_\xi(1)$
	\begin{align*}
			\overline{\psi_\alpha}(up)&=\frac{1}D\sum_{[\ga_i]\in Z_\xi\backslash \Ga_\xi}\psi_\alpha (\ga_i \cdot up)\\
			&=\frac{1}D\sum_{[\ga_i]\in Z_\xi\backslash \Ga_\xi}\psi_\alpha (\ga_i u \ga_i^{-1}\cdot \ga_ip)\\
			&=\frac{1}D\sum_{[\ga_i]\in Z_\xi\backslash \Ga_\xi}\psi_\alpha ( \ga_ip)\\
			&=	\overline{\psi_\alpha}(p),
	\end{align*}
where the third equality uses the $N_\xi$-invariance of $\psi_\alpha$ together with the fact that $P_\xi(1)$ normalizes $N_\xi$.

	To compute	$\overline{\psi_\alpha}(1)$ we can write $\ga_i$ uniquely as $n_im_i\in P_\xi(1)$, then by the $N_\xi$-invariance of $\psi_\alpha$, we obtain
	\begin{align*}
	\overline{\psi_\alpha}(1)&=\frac{1}N\sum_{[\ga_i]\in Z_\xi\backslash \Ga_\xi}\psi_\alpha (m_i)\\
	&=\frac{1}N\sum_i \Ad(m_i^{-1})\otimes \rho(m_i^{-1})\left((u_1^*\wedge\dots \wedge u_n^*)\otimes v_\alpha\right)\\
	&=(u_1^*\wedge\dots \wedge u_n^*)\otimes \left(\frac{1}N\sum_i \rho(m_i^{-1})(v_\alpha)\right),
	\end{align*}
where the last equality uses the fact that $K_\xi$ acts isometrically on $\HH_\xi(1)$ so in particular it preserves its volume form. Since $A_\xi$ commutes with $K_\xi$, it is clear that
\[v:=\left(\frac{1}N\sum_i \rho(m_i^{-1})(v_\alpha)\right)\in V_{-2}.\]
The uniqueness of $v$ follows from the injectivity of $i^*$.
Finally, to see why $\Theta v$ is fixed by $\Ga_\xi$, we first note that the collection of $\{m_i\}$ form a group. Thus $\rho(m_i)v=v$ for every $m_i$. Since $\Theta$ fixes the Lie algebra of $K_\xi$, it commutes with $\rho(m_i)$ for all $m_i$. It follows that $\rho(m_i)(\Theta v)=\Theta v$ for all $m_i$.   Since $\Theta v\in \mathfrak n_\xi$,  $\Ga_\xi$ fixes $\Theta v$.
\end{proof}

\begin{remark}\label{rem:non-toric}
	The proposition is an explicit realization of the isomorphism $H^n(\Ga_\xi,\Ad)\cong (H^n(Z_\xi,\Ad))^{\Ga_\xi}$ obtained for example via the spectral sequence. In the example of Kleinian group with infinitely many cusps constructed by Italiano, Martelli and Migliorini \cite{IMM}, they are toric. However, there are in general non-toric examples (See \cite{FKS}) where $\Ga_\xi$ does not fix any nontrivial vector in $\mathfrak n_\xi$, hence $H^n(\Ga_\xi,\Ad)=0$.
\end{remark}
%

\section{construction of Eisenstein series}
\label{sec:eisen}
In Section \ref{sec:nform}, we constructed $\Ga_{\xi}$-invariant V-valued ($V=\R$ or $\mathfrak{g}$) $n$-forms $\phi_{\xi, s}$ on $\H^{n+1}$ out from any full rank parabolic fixed point $\xi\in \partial_\infty\mathbb H^{n+1}$. In this section, we will further construct from each $\phi_{\xi,s}$ a $\Ga$-invariant $n$-form on  $\H^{n+1}$, i.e. an element in $\Omega^{n}(\H^{n+1},  V)^{\Ga}$, by the process of \emph{Eisenstein series}. Since we will need $\Ga$-actions on the differential forms, we want to first extend our definition of $\phi_{\xi,s}$ to a function that supports on the entire $G$, but of course this is uniquely determined because $\phi_{\xi,s}$ is already a form on $\H^{n+1}$. By the discussion in Section \ref{sec:differential forms} and in view of its value on $P_\xi$, the unique extension $\phi_{\xi,s}\in C^\infty\left(G, \Hom(\Lambda^n(\mathfrak a_\xi\oplus \mathfrak n_\xi), V)\right)$ satisfies for any $g\in G$,
\begin{enumerate}
	\item $\phi_{\xi,s}(g)=\phi_{\xi,s}(n\cdot a\cdot k)=\Ad_{\g}^*(k^{-1})\otimes \rho(k^{-1})(\phi_{\xi,s}(n))\cdot t_a^{s}$,
	\item $\phi_{\xi,s}(n)=\phi_{\xi,s}(1)=(u_1^*\wedge \cdots \wedge u_n^*)\otimes v_\xi$,
\end{enumerate}
where $g=nak\in N_\xi A_\xi K$ is the unique Iwasawa decomposition, and $v_\xi$ is any vector in $V_{-2}$ if $V=\g$. From Proposition \ref{prop:nontoric}, we know that any cohomology class in $H^n(\Ga_\xi,\Ad)$ can be represented by the above $\phi_{\xi,s}$ which is (left) $N_\xi$-invariant and $\Ga_\xi$-invariant. We now define 
$$E(\phi_{\xi}, g, s)=\sum_{\ga\in \Ga_{\xi}\backslash \Ga} \phi_{\xi, s}(\ga g)$$
to be the Eisenstein series associated with $\phi_\xi$ (with degree $s$). We will see later there is a unique $s$ such that $E(\phi_{\xi}, g, s)$ is closed so we will not emphasize on the dependence of $s$. By the construction, it is $\Ga$-invariant. But it is unclear whether it converges or not.

\begin{proposition}
\label{prop:convergence}
Let $\xi$ be a full rank parabolic fixed point. The Eisenstein series $E(\phi_{\xi}, g, s)$ absolutely converges if the Poincar\'e series $P_{s/2}(O):=\sum_{\gamma\in \Gamma} e^{-(s/2) d(O, \gamma O)}$ converges for some $O\in \mathbb{H}^{n+1}$. 
\end{proposition}

\begin{proof}
Fix any $g\in G$, by the Iwasawa decomposition, we have
$\ga g=n_{\ga}a_{\ga}k_{\ga}\in N_{\xi}A_{\xi}K$ where $K$ is the stabilizer of $O$. We can write each term 
$$\phi_{\xi, s}(\ga g)=\phi_{\xi, s}(n_{\ga}a_{\ga}k_{\ga})=\Ad^{\ast}_{\g}(k_{\ga}^{-1})\otimes \rho(k_{\ga}^{-1}) \phi_{\xi}(1) \cdot t_{\ga}^{s},$$
where we set $t_\ga=t(a_\ga)$. If we fix any norm on $\Hom(\Lambda^n(\mathfrak a_\xi\oplus \mathfrak n_\xi), V)$ then we can estimate
$$||E(\phi_{\xi}, g, s)||\leq \sum_{\ga\in \Ga_{\xi}\backslash \Ga} ||\phi_{\xi, s}(\ga g)||=\sum_{\ga\in \Ga_{\xi}\backslash\Ga}||\Ad^{\ast}_{\mathfrak{g}}(k_{\ga}^{-1})\otimes \rho(k_{\ga}^{-1}) \phi_{\xi}(1)|| \cdot t_{\ga}^{s}.$$
Since $K$ is compact, $\Ad^{\ast}_{\g}(k_{\ga}^{-1})\otimes \rho(k_{\ga}^{-1}) \phi_{\xi}(1)$ is uniformly bounded in $\Lambda^n\mathfrak{g}^{\ast}\otimes V$ hence also uniformly bounded when projected onto $\Lambda^n(\mathfrak a_\xi^*\oplus \mathfrak n_\xi^*)\otimes V$. So for the absolute convergence of the Eisenstein series, it suffices to consider the series
$$\sum_{\ga\in \Ga_{\xi}\backslash \Ga} t_{\ga}^{s}$$
and show its convergence. Note that the above series is well-defined since $t_\ga$ does not depend on the choice of $\ga$ in the coset $\Ga_\xi\backslash \Ga$. Indeed, if $\ga_0\in \Ga_\xi=\Ga\cap P_\xi$ is any element, then $\ga_0\in N_\xi K_\xi$ so that we can write $\ga_0=n_0m_0$. Since $m_0$ commutes with $a_\ga$ and normalizes $N_\xi$, we have $\ga_0\ga g=n_0m_0n_\ga a_\ga k_\ga=n_0(m_0n_\ga m_0^{-1})a_\ga m_0k_\ga$, which gives rises to the $N_\xi A_\xi K$--Iwasawa decomposition. This shows that the $A_\xi$ component does not change when replacing $\ga$ by $\ga_0\ga$, thus $t_\ga$ is independent on the choice of $\ga$ in the coset.

Next, we want to relate the series to the Poincar\'e series. Note that since $\xi $ is a full rank parabolic fixed point, $\Ga_\xi$ acts cocompactly on $\HH_\xi(1)$, so there exists a constant $C$ such that a fundamental domain of $\Ga_\xi\backslash \HH_\xi(1)$ is contained in the metric ball $B(O,C/2)$. Thus for any coset $\Ga_\xi\ga$, there exists a representative $\bar \ga\in \Ga_\xi\ga$ such that under the Iwasawa decomposition $\bar{\ga} g=n_\ga a_\ga k_\ga$, the unipotent component $n_\ga$ translates $O$ at most $C$, i.e. $d(n_\ga O, O)\leq C$. Now we can estimate the Poincar\'e series
\begin{align*}
P_{s}(gO)=\sum_{\ga\in \Ga}e^{-sd(O, \ga g O)}&\geq \sum_{[\bar \gamma]\in \Ga_{\xi}\backslash \Ga}e^{-sd(O,  n_{\ga} a_{\ga}O)}\\
&\geq \sum_{[\bar \gamma]\in \Ga_{\xi}\backslash \Ga}e^{-s(d(O,  n_\ga O)+d(n_\ga O, n_{\ga}a_{\ga}O))}\\
&\geq e^{-sC} \sum_{[\bar \ga]\in \Ga_{\xi}\backslash \Ga} e^{-s d(O, a_{\ga}O)},
\end{align*}
where the first inequality makes the particular choice of $\bar \ga$ described above, the second inequality uses the triangle inequality, and the last inequality uses the left invariance of the metric.

Observe that $a_\ga O$ is on the geodesic connecting $O$ and $\xi$. By hyperbolic geometric computation,  $d(O, a_{\ga}O)=|\ln t^{2}_{\ga}|$, and $t_\ga>1$ if and only if $a_\ga O$ lies inside the horoball of $\HH_\xi(1)$. In fact, the computation can be carried out within a totally geodesic copy $\mathbb H^2\subset \mathbb H^{n+1}$ which contains $\xi, O$ where $a_\ga$ acts by isometry. Then without loss of generality, we can use the upper half plane model and assume $O=i$, $\xi=i\infty$, and $a_\ga=\operatorname{diag}(a,a^{-1})\in \SL_2(\mathbb R)$. The mobius transformation gives $a_\ga i=a^2i$, hence $d(O,a_\ga O)=d(i,a^2i)=\ln|a^2|$. On the other hand, from the definition of the character \eqref{eq:character}, we know that $t_\ga=t(a_\ga)=a$ by the following matrix computation
\[\left(\begin{array}{cc}
	a &  \\
	& a^{-1}
\end{array}\right)\left(\begin{array}{cc}
0 & 1 \\
0 & 0
\end{array}
\right)\left(\begin{array}{cc}
	a &  \\
	& a^{-1}
\end{array}
\right)^{-1}=a^2\left(\begin{array}{cc}
	0 & 1 \\
	0 & 0
\end{array}\right).
\]

By the geometry of cusps (essentially the Margulis lemma), we know that there are only finitely many $\Ga_\xi\backslash \Ga$-orbits that lie inside any horoball $B(\xi)$ at $\xi$, that is, the cardinality of $\Ga_\xi\backslash \Ga O\cap B(\xi)$ (which makes sense since $B(\xi)$ is $\Ga_\xi$-invariant) is always finite. So are finitely many $\Ga_\xi\backslash \Ga g$-orbits ($\Ga_\xi\backslash \Ga gO\cap B(\xi)$) according to triangle inequality. Hence, up to a finite error, we have
$$P_{s}(gO)\geq e^{-sC} \sum_{[\ga]\in \Ga_{\xi}\backslash \Ga} e^{2s\ln t_\ga}= e^{-sC} \sum_{[\ga]\in \Ga_{\xi}\backslash \Ga} t_\ga^{2s}.$$
Note that $d(O,a_\ga O)=-2\ln(t_\ga)$ if and only if $t_\ga\leq 1$,  which holds for all but finitely many $[\ga]\in \Ga_\xi\backslash \Ga$, since all but finitely many $[\ga] gO$ lie outside the horoball of $\mathcal H_\xi(1)$. Hence up to passing finite many terms, the above inequality holds.

Note that the convergence of the Poincar\'e series doesn't depend on the basepoint. Thus, if $P_{s/2}(O)$ converges, then the Eisenstein series $E(\phi_{\xi}, g, s)$ converges.
\end{proof}

In view of Proposition \ref{prop:adclosed} and Corollary \ref{cor:closed} we obtain

\begin{corollary}\label{coro:closed1}
Let $\xi\in \H^{n+1}$ be any full rank parabolic fixed point.
\begin{enumerate}
	\item If $V=\mathfrak g$ is the adjoint representation, then the Eisenstein series $E(\phi_{\xi}, g, 2n+2)$ is always an absolutely convergent closed form.
	\item If $V=\R$ is the trivial representation, and in addition if either $\Ga$ is of convergence type or $\delta(\Ga)<n$, then the Eisenstein series $E(\phi_{\xi}, g, 2n)$ is an absolutely convergent closed form.
\end{enumerate}

\end{corollary}

\begin{proof}
The absolute convergence follows from Proposition \ref{prop:convergence}. To see it is closed, we have
$$d E(\phi_{\xi}, g, s)=d \left(\sum_{[\ga]\in \Ga_{\xi}\backslash \Ga} \phi_{\xi, s}(\ga g)\right)=\sum_{[\ga]\in \Ga_{\xi}\backslash \Ga} d \phi_{\xi, s}(\ga g)=0.$$
Here $\phi_{\xi, s}$ is closed by Corollary \ref{cor:closed} if $V=\mathfrak{g}$. If $V=\R$, $\phi_{\xi, s}$ is automatically closed since it is a top form. The absolute convergence of both series ensures the interchanging of the differential operator $d$ with the infinite sum in the second equality. The last equality follows from Lemma \ref{lem:commute}. 
\end{proof}

 
\section{Intertwining operators}
\label{sec:intertwining}

In this section, we assume that all the cusps are full rank, and $V=\R$ or $\mathfrak{g}$. Let $E(\phi_{\xi'}, g, s)$ be the Eisenstein series corresponding to the parabolic fixed point $\xi'$. In order to see which cohomology class the Eisenstein series $E(\phi_{\xi'}, g, s)$ restricts to in $H^n(\Ga_\xi,V)\cong H^n(\mathfrak u_\xi, V)$, we need to look at the image of $E(\phi_{\xi'}, g, s)$ under the map $I_2\circ r_1$ (See the following commutative diagram), which is equivalent to trace along $r_2\circ I_1$.

\[
\begin{tikzcd}
	\Omega^{\ast}(\H^{n+1}, V)^{\Ga} \arrow{r}{I_1} \arrow[swap]{d}{r_1} & \Omega^{\ast}(\H^{n+1}, V)^{N_{\xi}} \arrow{d}{r_2} \\%
	\Omega^{\ast}(\HH_{\xi}(1), V)^{\Ga_{\xi}} \arrow{r}{I_2}& \Omega^{\ast}(\HH_{\xi}(1), V)^{N_{\xi}}
\end{tikzcd}
\]

To obtain the image under $I_1$, we define the intertwining operator from $\xi'$ to $\xi$ as
$$E^{\xi}(\phi_{\xi'}, g, s)=\int_{u\in  \Gamma_{\xi} \backslash N_{\xi}} E(\phi_{\xi'}, ug, s) du,$$
where $du$ is the Haar measure normalized such that $\vol(\Ga_\xi\backslash N_\xi)=1$. Then by restricting the integral to the horosphere $\HH_{\xi}(1)$, we obtain an element in $\Omega^{\ast}(\HH_{\xi}(1), V)^{N_{\xi}}$ which is exactly the image of $E(\phi_{\xi'}, g, s)$ under $r_2\circ I_1$. The main goal of this section is to compute explicitly $E^{\xi}(\phi_{\xi'},g,s)$. We will follow the general approach in \cite[Chapter II]{HarishChandra}. For the convenience, we introduce the following notation: for any $g, h\in G$, we write $^{g}h$ to denote the conjugate of $h$ by $g$, that is, $\tensor[^{g}]{h}{}=ghg^{-1}$. We will need the following two lemmas.

\begin{lemma}
\label{t-part} For any $g\in G$ and $a\in A_{\xi}$, one has 
$$\phi_{\xi, s}(ag)=t_a^{s}\cdot \phi_{\xi, s}(g).$$

\end{lemma}

\begin{proof}
Let $g=n_{g}a_{g}k_{g}$. By definition, 
$$\phi_{\xi, s}(ag)=\phi_{\xi, s}(an_{g}a_{g}k_{g})=\phi_{\xi, s}(^{a}n_{g} aa_{g}k_{g})=t_{a}^{s}t_{g}^{s}\phi_{\xi}(^{a}n_{g}k_{g})=t_{a}^{s}t_{g}^{s}\phi_{\xi}(k_{g}).$$
where the last equality follows from the $N_{\xi}$-invariance of $\phi_{\xi}$. We also have
$$\phi_{\xi, s}(g)=\phi_{\xi, s}(n_{g}a_{g}k_{g})=t_{g}^{s}\phi_{\xi}(n_{g}k_{g})=t_{g}^{s}\phi_{\xi}(k_{g}). $$
Therefore, $\phi_{\xi, s}(ag)=t_{a}^{s}\phi_{\xi, s}(g).$
\end{proof}

\begin{lemma}(Geometric Bruhat decomposition)\label{lem:Bruhat} Let $G=\SO^+(n+1,1)$ and fix a basepoint $O\in \mathbb H^{n+1}$. For any $\xi, \xi'\in \partial_\infty \mathbb H^{n+1}$, $G$ decomposes as a disjoint union
	\begin{equation*}
		G= \begin{cases}
			P_\xi\cup N_\xi w P_\xi &\text{if } \xi=\xi'\\
			kP_\xi \cup N_{\xi'}kwP_\xi&\text{if } \xi\neq\xi'
		\end{cases},
	\end{equation*}	
	where $w\in K$ is any isometry that reflects the geodesic $O\xi$, and $k\in K$ is any isometry that sends $\xi$ to $\xi'$.
\end{lemma}

\begin{proof}
	The case $\xi=\xi'$ is just the classical Bruhat decomposition. For $\xi\neq\xi'$, and for any $k\in K$ that sends $\xi$ to $\xi'$, using the classical Bruhat decomposition we have,
	$$G=kG=kP_{\xi}\cup k N_{\xi} w P_{\xi}=kP_{\xi}\cup N_{\xi'} k wP_{\xi},$$
	where the last equality uses $kN_{\xi}k^{-1}=N_{\xi'}$. 
\end{proof}	

Although we started with $\phi_{\xi,s}$ whose value is in $C^\infty(P_\xi, \Hom(\Lambda^n \mathfrak n_\xi, V))$, after taking the Eisenstein series and intertwining operator, even though it is still $N_\xi$-invariant, the value could lie in a bigger space $C^\infty(P_\xi, \Hom(\Lambda^n (\mathfrak n_\xi\oplus \mathfrak a_\xi), V))$. For this reason, we introduce the extended space $\overline{\Omega}^{n}(\HH_{\xi}(1), V)\supset \Omega^{n}(\HH_{\xi}(1), V)$ to be the restriction of $V$-valued forms on $\H^{n+1}$ to the horosphere $\HH_{\xi}(1)$, without projecting to its (co)tangent space. In other words, forms in $\overline{\Omega}^{n}(\HH_{\xi}(1), V)$ may look like $(dy\wedge dx_1\wedge\dots \wedge dx_{n-1})\otimes v$. Now we can state our theorem.

\begin{theorem}\label{thm:intertwining}
Let $E(\phi_{\xi}, g, s), E(\phi_{\xi'}, g, s)$ denote the Eisenstein series corresponding to the rank $n$ parabolic fixed points $\xi, \xi'$ respectively. Then for any $g\in P_\xi$, we have
$$E^{\xi}(\phi_{\xi'}, g, s)=\epsilon\phi_{\xi', s}(g)+c_s(\phi_{\xi'})_{-s+2n}(g)$$
where
\begin{equation*}
	\epsilon= \begin{cases}
		0 &\text{if $\Ga_{\xi}$ and $\Ga_{\xi'}$ are not $\Ga$-conjugate}\\
		1 &\text{if $\xi=\xi'$}\\
	\end{cases},
\end{equation*}
and $c_s: \Omega^{n}(\HH_{\xi'}(1), V)^{N_{\xi'}}\rightarrow \overline{\Omega}^{n}(\HH_{\xi}(1), V)^{N_\xi}$ is a linear operator.
\end{theorem}

\begin{proof}
By the definition, 
$$E^{\xi}(\phi_{\xi'}, g, s)=\int_{u\in  \Ga_{\xi} \backslash N_{\xi}} E(\phi_{\xi'}, ug, s)du=\sum_{\gamma\in \Ga_{\xi'}\backslash \Gamma} \int_{u\in \Ga_{\xi}\backslash N_{\xi}} \phi_{\xi', s}(\gamma u g)du.$$
By assumptions, the Eisenstein series is absolutely convergent and this ensures the interchanging between the integral and the infinite summation. This shows up several times in the rest computations, but we emphasize that all the summations appear are absolutely dominated by the Eisenstein series (which is absolutely convergent), thus interchanging the summation with the integral simply causes no problem by Fubini's theorem.

We first prove the case when $\xi=\xi'$. By the Bruhat decomposition, we have
$$G=SO^{+}(n+1, 1)=P_{\xi}\cup N_{\xi}w P_{\xi}.$$
Hence, elements in $\Gamma$ is either represented by elements in $P_{\xi}$ or by the ones in $N_{\xi}wP_{\xi}$. We denote the set of former elements by $\Ga_{1}$ which is exactly $\Ga_{\xi}$ and the set of latter one by $\Ga_{w}$. Now we split the summation of $E^{\xi}(\phi_{\xi}, g, s)$ into two kinds,
\begin{align*}
	E^{\xi}(\phi_{\xi}, g, s)&=\sum_{\gamma\in \Ga_{\xi}\backslash \Gamma_1} \int_{u\in \Ga_{\xi}\backslash N_{\xi}} \phi_{\xi, s}(\gamma u g)du+\sum_{\gamma\in \Ga_{\xi}\backslash \Gamma_w} \int_{u\in \Ga_{\xi}\backslash N_{\xi}} \phi_{\xi, s}(\gamma u g)du\\
	&=\mathfrak E_1+\mathfrak E_2,
\end{align*}
where $ \Ga_{\xi}\backslash \Ga_w$ makes sense since $\Ga=\Ga_\xi\cup \Ga_w$ and it denotes the set of all non-trivial right cosets of $\Ga_\xi$ in $\Ga$. It is clear that 
$$\mathfrak E_1=\int_{u\in \Ga_{\xi}\backslash N_{\xi}} \phi_{\xi, s}(ug)du=\phi_{\xi, s}(g),$$
where the last equality follows from the $N_{\xi}$-invariance of $\phi_{\xi, s}$. To simplify $\mathfrak E_2$, we will need the following lemma.

\begin{lemma}\label{lem:doublecoset}
	$\Ga_\xi$ acts on $\Ga_\xi\backslash \Ga_w$ from the right with trivial stabilizer, hence the quotient is the double coset  $\Ga_{\xi}\backslash \Ga_{w} / \Ga_{\xi}$.
\end{lemma}
\begin{proof}
	Given any $\ga_w\in \Ga_w$ and $a\in \Ga_\xi$ such that $\Ga_\xi \ga_w\cdot a=\Ga_\xi\ga_w$, it suffices to show $a=1$. From the identity, we obtain $\ga_w a \ga_w^{-1}\in \Ga_\xi$, or that $a\in \Ga_{\ga_w\xi}$. Hence, 
	\[a\in \Ga_{\xi}\cap  \Ga_{\ga_w\xi}.\]
	Since $\ga_w\notin \Ga_\xi$, we know $\ga_w\xi\neq \xi$ hence $\Ga_{\xi}\cap  \Ga_{\ga_w\xi}=1$. Therefore $a=1$ and the action has trivial stabilizer.
\end{proof}
Now we can simplify $\mathfrak E_2$ as
\begin{align*}
	\sum_{\ga\in  \Ga_{\xi}\backslash \Ga_{w}}\int_{ u\in \Ga_{\xi}\backslash N_{\xi}}\phi_{\xi, s}(\ga ug)du&=\sum_{\ga\in \Ga_{\xi}\backslash \Ga_{w} / \Ga_{\xi}} \sum_{\ga_\xi\in \Ga_\xi}\int_{ u\in \Ga_{\xi}\backslash N_{\xi}}\phi_{\xi, s}(\ga \ga_\xi ug)du\\
	&=\sum_{\ga\in \Ga_{\xi}\backslash \Ga_{w} / \Ga_{\xi}} \int_{u\in N_{\xi}} \phi_{\xi, s}(\ga ug)du,
\end{align*}
where the last equality uses the fact that $N_\xi$ is abelian and $du$ is the bi-invariant Haar measure. We compute further $\mathfrak E_2$.

{\small
\begin{align*}
\mathfrak E_2 &= \sum_{\gamma\in \Ga_{\xi}\backslash \Ga_{w}/ \Ga_{\xi}} \int_{u\in N_{\xi}} \phi_{\xi, s}(\gamma u g) du \\
&=  \sum_{\gamma\in \Ga_{\xi}\backslash \Ga_{w}/ \Ga_{\xi}} \int_{u\in N_{\xi}} \phi_{\xi, s}(u_{\gamma} w p_{\gamma} u g)du && \gamma=u_{\gamma} wp_{\ga}\in N_\xi w P_\xi  \\
&=   \sum_{\gamma\in \Ga_{\xi}\backslash \Ga_{w}/ \Ga_{\xi}} \int_{u\in N_{\xi}} \phi_{\xi, s}(wp_{\ga} u n_{g} a_{g} m_{g} ) du && \phi_{\xi, s} \text{ is } N_{\xi} \text{-invariant} \\
&=   \sum_{\gamma\in \Ga_{\xi}\backslash \Ga_{w}/ \Ga_{\xi}} \int_{u\in N_{\xi}} \phi_{\xi, s}(wm_{\ga} a_{\ga} n_{\ga} u n_{g} a_{g} m_{g} ) du &&  p_\ga =m_\ga a_\ga n_\ga\in K_\xi A_\xi N_\xi\\
&= \sum_{\gamma\in \Ga_{\xi}\backslash \Ga_{w}/ \Ga_{\xi}} \int_{u'\in N_{\xi}} \phi_{\xi, s}(wm_{\ga} a_{\ga} u'a_{g} m_{g} ) du' &&  \text{set } u'=n_\ga u n_g, du'=du\\
	&=  \sum_{\gamma\in \Ga_{\xi}\backslash \Ga_{w}/ \Ga_{\xi}} \int_{u'\in N_{\xi}} \phi_{\xi, s}(w a_{\ga} a_{g} m_\ga  \tensor[^{a_{g}^{-1}}]{(u')}{} m_g) du' &&   K_\xi \text{ commutes with } A_\xi \\
		&=  \sum_{\gamma\in \Ga_{\xi}\backslash \Ga_{w}/ \Ga_{\xi}} \int_{u''\in N_{\xi}} \phi_{\xi, s}(w a_{\ga} a_{g} m_\ga u'' m_g) du'' \cdot t_g^{2n} &&   \text{set } u''= \tensor[^{a_{g}^{-1}}]{(u')}{}, du''=t_g^{-2n}du'\\
         &= \sum_{\gamma\in \Ga_{\xi}\backslash \Ga_{w}/ \Ga_{\xi}}  \int_{u''\in N_{\xi}} \phi_{\xi, s}(a_{\ga}^{-1} a_{g}^{-1} w  m_\ga u'' m_{g}) du''\cdot t_g^{2n} && w \text{ inverses } A_\xi \\
         &=     \sum_{\gamma\in \Ga_{\xi}\backslash \Ga_{w}/ \Ga_{\xi}}  \int_{u''\in N_{\xi}} \phi_{\xi,s}(w m_\ga u'' m_{g}   ) du'' t_{\ga}^{-s} t_{g}^{2n-s} &&  \text{by Lemma }\ref{t-part} \\
         &=\Ad_{\mathfrak{g}}^{*}(m_g^{-1})\otimes \rho(m_g^{-1})\left(\sum_{\gamma\in \Ga_{\xi}\backslash \Ga_{w}/ \Ga_{\xi}}  \int_{u''\in N_{\xi}} \phi_{\xi,s}(w m_{\ga}u'') d u'' t_{\ga}^{-s} \right)t_{g}^{2n-s}\\
         &=\Ad_{\mathfrak{g}}^{*}(m_g^{-1})\otimes \rho(m_g^{-1})(C_0(\phi_{\xi,s}))\cdot t_{g}^{2n-s},
\end{align*}
}%
where in the last line $C_0(\phi_{\xi,s})$ is just a constant (independent of $g$) in $\Hom(\Lambda^n(\mathfrak a_\xi\oplus \mathfrak p_\xi),V)$.  Hence, if we write the variable $g=n_g a_g m_g$, then we can regard $\mathfrak E_2$ as the differential form constructed via the extension (with degree $2n-s$) of the $N_\xi$-invariant form whose initial value is exactly $C_0(\phi_{\xi,s})$. We denote  $c_s: \Omega^{n}(\HH_{\xi}(1), V)^{N_{\xi}}\rightarrow \overline{\Omega}^{n}(\HH_{\xi}(1), V)^{N_\xi}$ the unique operator given by the initial value $c_s(\phi_\xi)(1)=C_0(\phi_{\xi,s})$. From the expression, it is clear that $c_s$ is linear. Therefore we obtain,
\begin{align*}
	E^{\xi}(\phi_{\xi}, g, s)=\int_{u\in \Ga_{\xi}\backslash N_{\xi}} E(\phi_\xi, ug, s) du&=\phi_\xi(n_g m_g) t_{g}^{s}+ c_s(\phi_\xi)(n_g m_g) t_{g}^{2n-s}\\
	&=\phi_{\xi, s}(g)+c_s(\phi_\xi)_{(2n-s)}(g)
\end{align*}



Next, we assume $\xi$ and $\xi'$ are not $\Ga$-conjugate. Then similar to Lemma \ref{lem:doublecoset}, we have,

\begin{lemma}
	Suppose $\xi$ and $\xi'$ are not $\Ga$-conjugate, then $\Ga_\xi$ acts on $\Ga_{\xi'}\backslash \Ga$ from the right with trivial stabilizer, hence the quotient is the double set  $\Ga_{\xi'}\backslash \Ga/ \Ga_{\xi}$.
\end{lemma}

\begin{proof}
	For any $\ga\in \Ga$ and $a\in \Ga_\xi$ such that $\Ga_{\xi'} \ga\cdot a=\Ga_{\xi'}\ga$, we have $\ga a \ga^{-1}\in \Ga_{\xi'}$, or that $a\in \Ga_{\ga\xi'}$. Hence, 
	\[a\in \Ga_{\xi}\cap  \Ga_{\ga\xi'}.\]
	Since $\xi$ and $\xi'$ are not $\Ga$-conjugate, we know $\ga\xi'\neq \xi$ hence $\Ga_{\xi}\cap  \Ga_{\ga\xi'}=1$. Therefore $a=1$ and the action has trivial stabilizer.
\end{proof}

Using the lemma, we can simplify 
\begin{align*}
	 E^{\xi}(\phi_{\xi'}, g, s)&=\sum_{\gamma\in \Ga_{\xi'}\backslash \Ga} \int_{u\in  \Ga_{\xi}\backslash N_{\xi}} \phi_{\xi', s}(\gamma u g)du\\
	 &= \sum_{\ga\in \Ga_{\xi'}\backslash \Ga/ \Ga_{\xi}} \int_{u\in  N_{\xi}} \phi_{\xi', s}(\gamma u g)du\\
	 &=\mathfrak E.
\end{align*}
Also by the Bruhat decomposition (Lemma \ref{lem:Bruhat}), for any $\gamma\in \Ga$, either $\gamma\in kP_\xi$ or $\gamma\in N_{\xi'}kwP_{\xi}$. If $\gamma\in kP_\xi$, then $k^{-1}\gamma\in P_\xi$, so $k^{-1}\gamma\xi=\xi$, or that $\gamma\xi=k\xi=\xi'$, contradiction to that $\xi$ and $\xi'$ are not $\Ga$-conjugate. Thus $\gamma\in N_{\xi'}kwP_{\xi}$. We now further compute $\mathfrak E$.
{\small
\begin{align*}
\mathfrak E &=\sum_{\ga\in \Ga_{\xi'}\backslash \Ga/ \Ga_{\xi}} \int_{u\in  N_{\xi}} \phi_{\xi', s}(\gamma u g)du\\
&=\sum_{\ga\in \Ga_{\xi'}\backslash \Ga/ \Ga_{\xi}}  \int_{u\in N_{\xi}} \phi_{\xi', s}(n'_{\gamma} k w p_{\gamma}  u g)du &&  \ga=n'_{\ga}kwp_{\ga}\in N_{\xi'}kw P_{\xi}\\
 &=\sum_{\ga\in \Ga_{\xi'}\backslash \Ga/ \Ga_{\xi}}  \int_{u\in N_{\xi}}  \phi_{\xi', s}(k w p_{\gamma} ug) du && \phi_{\xi', s} \text{ is } N_{\xi'}\text{-invariant}\\
 &=   \sum_{\gamma\in \Ga_{\xi'}\backslash \Ga/ \Ga_{\xi}} \int_{u\in N_{\xi}} \phi_{\xi', s}(kwm_{\ga} a_{\ga} n_{\ga} u n_{g} a_{g} m_{g} ) du &&  \text{(1)}\\
 &= \sum_{\gamma\in \Ga_{\xi'}\backslash \Ga/ \Ga_{\xi}} \int_{u'\in N_{\xi}} \phi_{\xi', s}(kwm_{\ga} a_{\ga} u'a_{g} m_{g} ) du' &&  \text{set } u'=n_\ga u n_g, du'=du\\
 &=  \sum_{\gamma\in \Ga_{\xi'}\backslash \Ga/ \Ga_{\xi}} \int_{u'\in N_{\xi}} \phi_{\xi', s}(kw a_{\ga} a_{g} m_\ga  \tensor[^{a_{g}^{-1}}]{(u')}{} m_g) du' &&   K_\xi \text{ commutes with } A_\xi \\
 &=  \sum_{\gamma\in \Ga_{\xi'}\backslash \Ga/ \Ga_{\xi}} \int_{u''\in N_{\xi}} \phi_{\xi', s}(kw a_{\ga} a_{g} m_\ga u'' m_g) du'' \cdot t_g^{2n} &&   \text{set } u''= \tensor[^{a_{g}^{-1}}]{(u')}{}, du''=t_g^{-2n}du'\\
 &= \sum_{\gamma\in \Ga_{\xi'}\backslash \Ga/ \Ga_{\xi}}  \int_{u''\in N_{\xi}} \phi_{\xi', s}(\tensor[^{k}]{(a_{\ga}^{-1} a_{g}^{-1})}{} k w  m_\ga u'' m_{g}) du''\cdot t_g^{2n} && w \text{ inverses } A_\xi \\
 &=     \sum_{\gamma\in \Ga_{\xi'}\backslash \Ga/ \Ga_{\xi}}  \int_{u''\in N_{\xi}} \phi_{\xi',s}(kw m_\ga u'' m_{g}   ) du'' t_{\ga}^{-s} t_{g}^{2n-s} &&  \text{(2)}\\
 &=\Ad^*(m_g^{-1})\otimes \rho(m_g^{-1})\left(\sum_{\gamma\in \Ga_{\xi'}\backslash \Ga_{w}/ \Ga_{\xi}}  \int_{u''\in N_{\xi}} \phi_{\xi',s}(kw m_{\ga}u'') d u'' t_{\ga}^{-s} \right)t_{g}^{2n-s} && (3)\\
 &=\Ad^*(m_g^{-1})\otimes \rho(m_g^{-1})(C_0(\phi_{\xi',s}))\cdot t_{g}^{2n-s},
\end{align*}
}%

where in (1), we apply the Langlands decompositions $p_{\gamma}=m_{\ga}a_{\ga}n_{\ga}\in M_{\xi}A_{\xi}N_{\xi}, g=n_{g}a_{g}m_{g}\in N_{\xi}A_{\xi}M_{\xi}$.  In (2), we use $k a_{\ga}^{-1}a_{g}^{-1} k^{-1}\in A_{\xi'}$ and the fact that $k$ is an isometry which send character $t_\xi$ to $t_{\xi'}$. Then we apply Lemma \ref{t-part}. In (3), we use the property of $\phi_{\xi,s}$ and the fact that $m_g\in K$. Hence, if we denote  $c_s: \Omega^{n}(\HH_{\xi'}(1), V)^{N_{\xi'}}\rightarrow \overline{\Omega}^{n}(\HH_{\xi}(1), V)^{N_\xi}$ the unique operator given by the initial value $c_s(\phi_{\xi'})(1)=C_0(\phi_{\xi',s})$, then
$E^{\xi}(\phi_{\xi'}, g, s)=c_s(\phi_{\xi'})(n_{g}m_{g})t_{g}^{-s+2n}$. 
\end{proof}




\section{cohomology classes associated to cusps}\label{sec:proof}
In this section, we study the cohomology class of the restriction of the  Eisenstein series $E(\phi_{\xi'}, g, s)$ to the cusp associated to the parabolic fixed point $\xi$. Throughout the section we assume $s=2n+2$ if $V=\mathfrak{g}$ and $s=2n$ If $V=\R$. We also assume the Eisenstein series converges at $s$. By Corollary \ref{coro:closed1}, we have $d E(\phi_{\xi'}, g, s)=0$. Then 
$$d E^{\xi}(\phi_{\xi'}, g, s)=d \int_{u\in  \Gamma_{\xi} \backslash N_{\xi}} E(\phi_{\xi'}, ug, s) du=\int_{u\in  \Gamma_{\xi} \backslash N_{\xi}} d  E(\phi_{\xi'}, ug, s) du=0.$$
As before, the absolute convergence of the Eisenstein series  ensures the interchanging of differential and integral, and the last equality follows from Lemma \ref{lem:commute}. We will use the computation of $E^{\xi}(\phi_{\xi}, g, s)$ in Section \ref{sec:intertwining} to see which cohomology class $[E^{\xi}(\phi_{\xi'}, g, s)]|_{\HH_{\xi}(1)}$ represents in $H^n(\Ga_\xi,V)$.

\begin{proposition}\label{prop:cohomology}
Let $\xi$ and $\xi'$ be full rank parabolic fixed points. Then 
$$[E^{\xi}(\phi_{\xi'}, g, s)]|_{\HH_{\xi}(1)}=[\phi_{\xi}]\in H^{n}( \Ga_{\xi}, V)$$
 if $\xi=\xi'$. If $\xi$ and $\xi'$ are not $\Ga$-conjugate, then  
$$[E^{\xi}(\phi_{\xi'}, g, s)]|_{\HH_{\xi}(1)}=0\in H^{n}(\Ga_{\xi}, V).$$ 

\end{proposition}

\begin{proof}
By Theorem \ref{thm:intertwining}, 
$$E^{\xi}(\phi_{\xi'}, g, s)=\epsilon\phi_{\xi', s}(g)+c_{s}(\phi_{\xi'})_{-s+2n}(g). $$
 Since $d E^{\xi}(\phi_{\xi'}, g, s)=0$, and $d \phi_{\xi', s}(g)=0$, we obtain 
$$d(c_{s}(\phi_{\xi'})_{-s+2n})(g)=0.$$
Hence, it suffices to prove that
$$[c_{s}(\phi_{\xi'})_{-s+2n}(g)]|_{\HH_{\xi}(1)}=0\in H^{n}(\Ga_{\xi}, V).$$

We first consider the case that $V=\mathfrak{g}$. Recall that we can regard the extended form $c_{s}(\phi_{\xi'})\in \overline{\Omega}(\HH_{\xi}(1), V)^{N_\xi}$ as a function in $C^{\infty}(P_{\xi}(1), \Hom(\Lambda^{n}(\mathfrak a_\xi\oplus\mathfrak{n}_{\xi}), V))$. Since $V=V_{-2}\oplus V_{0}\oplus V_{2}$ and $c_{s}(\phi_{\xi'})$ is $N_{\xi}$-invariant,
we can decompose $c_{s}(\phi_{\xi'})(1)$ as
$$c_{s}(\phi_{\xi'})(1)=A_{2}+A_{0}+A_{-2}+\dfrac{dt}{t}  \wedge A'_{2} + \dfrac{dt}{t} \wedge A'_{0}+\dfrac{dt}{t} \wedge A'_{-2}, $$
where $A_{i}\in \Lambda^{n}\mathfrak{n}^{\ast}_{\xi}\otimes V_{i}$, and $A'_{i}\in \Lambda^{n-1} \mathfrak{n}^{\ast}_{\xi}\otimes V_{i}$ for $i\in \{-2, 0, 2\}$. For each $i$, we construct an $N_{\xi}$-invariant $V$-valued form $c_{i}(\phi_{\xi'})$ (or $c'_{i}(\phi_{\xi'})$) on the horosphere $\HH_{\xi}(1)$ whose initial value is $A_i$ (or $A'_{i}$), that is, $c_{i}(\phi_{\xi'})\in C^{\infty}(P_{\xi}(1), \Hom(\Lambda^{n} \mathfrak{n}_{\xi}, V))$ satisfies:
\begin{enumerate}
\item $c_{i}(\phi_{\xi'})(n)=c_{i}(\phi_{\xi'})(1)=A_{i}$,
\item $c_{i}(\phi_{\xi'})(pm)=\Ad^{\ast}_{\mathfrak{a}_{\xi}\oplus \mathfrak n_\xi}(m^{-1})\otimes \rho(m^{-1})(c_{i}(\phi_{\xi'})(p)) \text{ for all } m\in K_{\xi} \text{ and } p\in P_{\xi}(1). $
\end{enumerate}
And $c'_{i}(\phi_{\xi'})$ is defined similarly. Following the construction in Section \ref{sec:nform}, we can define degree $2n-s$ extensions of $c_{i}(\phi_{\xi'})$ and $c'_{i}(\phi_{\xi'})$  in $C^{\infty}(P_{\xi}, \Hom(\Lambda^{n} \mathfrak{p}_{\xi}, V))$, denoted by $(c_{i}(\phi_{\xi'}))_{2n-s}$ and $(c'_{i}(\phi_{\xi'}))_{2n-s}$ respectively. It follows that
$$c_{s}(\phi_{\xi'})_{2n-s}=\sum_{i} (c_{i}(\phi_{\xi'}))_{2n-s}+\sum_{i}\dfrac{dt}{t}\wedge (c'_{i}(\phi_{\xi'}))_{2n-s}.$$
By Proposition \ref{prop:adclosed},
\begin{align*}
0=d(c_{s}(\phi_{\xi'})_{2n-s})&=(2n-s+2-2n) \dfrac{dt}{t} \wedge c_{2}(\phi_{\xi'})_{2n-s}+(2n-s-2n) \dfrac{dt}{t} \wedge c_{0}(\phi_{\xi'})_{2n-s}\\
 &+(2n-s-2-2n)  \dfrac{dt}{t}\wedge c_{-2}(\phi_{\xi'})_{2n-s} +\sum_{i}\dfrac{dt}{t} \wedge (dc'_{i}(\phi_{\xi'})_{2n-s}),
\end{align*}
where $i\in \{-2, 0, 2\}$. Observe that 
$$ \dfrac{dt}{t} \wedge c_{i}(\phi_{\xi'})_{2n-s}(1)\in  \dfrac{dt}{t}\wedge \Hom(\Lambda^{n} \mathfrak{n}_{\xi}, V_{i})$$
and by the proof of Lemma \ref{lem:lie alg cohomology},
$$ \dfrac{dt}{t}\wedge d(c'_{i}(\phi_{\xi'})_{2n-s})(1)\in \dfrac{dt}{t}\wedge  \Hom(  \Lambda^{n} \mathfrak{n}_{\xi}, V_{i+2})$$
for $i\in \{-2, 0, 2\}$. 
Therefore, by comparing their components in $V_{2}, V_{0},$ and  $V_{-2}$ respectively, we obtain
$$  \dfrac{dt}{t}\wedge (2n-s+2-2n) c_{2}(\phi_{\xi'})_{2n-s}(1)+  \dfrac{dt}{t}\wedge d(c'_{0}(\phi_{\xi'}))_{2n-s}(1)=0,$$
$$\dfrac{dt}{t}\wedge (2n-s-2n)c_{0}(\phi_{\xi'})_{2n-s}(1) + \dfrac{dt}{t}\wedge d(c'_{-2}(\phi_{\xi'}))_{2n-s}(1)=0,$$
$$ \dfrac{dt}{t} \wedge (2n-s-2-2n)c_{-2}(\phi_{\xi'})_{2n-s}(1)=0.$$
Note that  $s=2+2n$, the coefficients in the above equalities are nonzero. Thus we have
$$c_{2}(\phi_{\xi'})(1)=\dfrac{d(c'_{0}(\phi_{\xi'}))(1)}{2-s},$$
$$c_{0}(\phi_{\xi'})(1)=\dfrac{d(c'_{-2}(\phi_{\xi'}))(1)}{-s},$$
$$c_{-2}(\phi_{\xi'})(1)=0.$$
By Lemma \ref{lem:commute}, 
$$c_{2}(\phi_{\xi'})=\dfrac{d(c'_{0}(\phi_{\xi'}))}{2-s},$$
$$c_{0}(\phi_{\xi'})=\dfrac{d(c'_{-2}(\phi_{\xi'}))}{-s},$$
$$c_{-2}(\phi_{\xi'})=0.$$
We see that $c_{0}(\phi_{\xi'})+c_{2}(\phi_{\xi'})$ is a coboundary, which equals the restriction (in the strong sense, i.e. also project the (co)tangent space) of $c_{s}(\phi_{\xi'})_{2n-s}$ to the horosphere $\HH_{\xi}(1)$. Hence 
$$[c_{s}(\phi_{\xi'})_{-s+2n}]|_{\HH_{\xi}(1)}=0\in H^n(\Ga_\xi, \Ad).$$
The argument is similar for the case that $V=\R$. In this case, we write
$$c_{s}(\phi_{\xi'})=c_{1}(\phi_{\xi'})+\dfrac{dt}{t}\wedge c_{2}(\phi_{\xi'})$$
where $c_{1}(\phi_{\xi'})\in C^{\infty}(P_{\xi}(1), \Hom(\Lambda^{n}\mathfrak{n}_{\xi}, \R))$ and $c_{2}(\phi_{\xi'})\in C^{\infty}(P_{\xi}(1),  \Hom(\Lambda^{n-1}\mathfrak{n}_{\xi}, \R))$. By Proposition \ref{prop:adclosed}, 
$$0=d(c_{s}(\phi_{\xi'})_{2n-s})=(2n-s-2n)\dfrac{dt}{t}\wedge c_{1}(\phi_{\xi'})_{2n-s},$$
which implies that $c_{1}(\phi_{\xi'})=0$. Hence
$$[c_{s}(\phi_{\xi'})_{-s+2n}]|_{\HH_{\xi}(1)}=0\in H^n(\Ga_\xi, \R).$$
\end{proof}

Now we are ready to prove Theorem \ref{thm:trivial}, \ref{thm:Ad} and Corollary \ref{cor:cusp}.

\subsection*{Proof of Theorem  \ref{thm:trivial}, \ref{thm:Ad}}Since every full rank parabolic subgroup corresponds to a parabolic fixed point $\xi\in \H^{n+1}$, it follows immediately from Proposition \ref{prop:cohomology}. Also, the harmonicity of the Eisenstein series $E(\phi_{\xi})$ in the case of trivial coefficient is clear since it is the (absolutely convergent) sum of harmonic forms (See Remark \ref{rem:harmonic}).\qed


\subsection*{Proof of Corollary \ref{cor:cusp}} Let $C_1,..., C_N$ be $N$ toric cusps. For each cusp $C_i$, choose a corresponding parabolic subgroup $\Ga_i\cong \mathbb Z^n$ (unique up to conjugate). By Lemma \ref{lem:lie alg cohomology} and Theorem \ref{thm:vanEst}, the dimension of the cohomology group $H^n(\Ga_i,\Ad)$ is $n$. Thus, in view of Theorem \ref{thm:Ad}, they in total correspond to $nN$ linearly independent Eisenstein cohomology classes in $H^n(\Ga,\Ad)$. So the corollary follows immediately.\qed

\section{Further discussions}

 Our work seems to suggest that the classical method of Eisenstein series should also fit in a broader context for certain problems, and this paper only focuses on a specific aspect of that, namely the cusp counting problem for hyperbolic manifolds. Before discussing possible directions of extensions of our results, we first give some examples where our theorems can be applied to. Nevertheless, we would like to point out that the dimension of group cohomology is often very hard to compute, so it is unclear how sharp the inequality in Corollary \ref{cor:cusp} is in general.

\subsection*{Non-uniform lattices} As is mentioned in the introduction, non-uniform lattices do not satisfy the assumption and conclusion of Theorem \ref{thm:trivial}. In this case, $\delta(\Gamma)=n$, $\Ga$ is of divergent type, and cohomology classes (with trivial coefficient) arising from cusps \emph{cannot} be linearly independent. However, our Theorem \ref{thm:Ad} and Corollary \ref{cor:cusp} both apply, thus giving an explicit upper bound on the number of cusps. One type of explicit constructions of lattices comes from arithmetic (e.g. take $\Ga=\SO(n+1,1;\mathbb Z)<\SO(n+1,1)$ up to finite index to kill the torsion). The number of cusps thus is closely related to the arithmetic feature of the construction (e.g. the ideal class group of the corresponding number field). There are also non-arithmetic constructions of lattices due to the work of Gromov--Piatetski-Shapiro \cite{GPS}, and our results are potentially more useful in these examples.

\subsection*{Geometrically finite Kleinian groups} The simplest such example arises from the Schottky-type construction. Take a rank $n$ parabolic subgroup $\Ga_1\cong \mathbb Z^n<\operatorname{Isom}^+(\mathbb H^{n+1})$ and an elementary subgroup generated by a single hyperbolic translation $\Ga_2\cong \mathbb Z<\operatorname{Isom}^+(\mathbb H^{n+1})$, then up to a choice of conjugates of $\Ga_1, \Ga_2$, the group $\Ga:=<\Ga_1,\Ga_2>$ is isomorphic to the free amalgamation $\Ga_1* \Ga_2$ by Maskit's Klein combination theorem \cite{Maskit}. The resulting Kleinian group is then geometrically finite and have exactly one cusp. On the other hand, the critical exponent satisfies $\delta(\Ga)<n$, so in this case both our Theorems hold. It is clear that the $n$-th betti number is $1$, but $H^n(\Ga, \Ad)$ will depend on the representation of $\Ga$ in $\operatorname{Isom}^+(\mathbb H^{n+1})$. Thus, our corollary gives a uniform lower bound on  $\dim H^n(\Ga, \Ad)$, i.e., $\dim H^n(\Ga, \Ad)\geq n$, and in fact it holds for all geometrically finite Kleinian groups $\Ga\cong \mathbb Z^n*\mathbb Z$.

\subsection*{Geometrically infinite Kleinian groups} As we mentioned in Remark \ref{rem:non-toric}, the Kleinian groups $\Ga<\operatorname{Isom}^+(\mathbb H^{n+1})$ ($4\leq n\leq 7$) constructed in \cite{IMM} has infinitely many full rank toric cusps. It is a finitely generated (also finitely presented in the case $n=6, 7$) normal subgroup of a lattice, which is constructed from an algebraic fibration of the lattice over $\mathbb Z$. It follows that $\delta(\Ga)=n$ \cite{Roblin}. As a consequence of Corollary \ref{cor:cusp}, the cohomology group $H^{n}(\Ga, \Ad)$ has infinite dimension.  On the other hand, Kapovich's example \cite{Kapovich1} of Kleinian group $\Ga<\operatorname{Isom}^+(\mathbb H^{4})$ has infinitely many rank $1$ cusps, and the critical exponent satisfies $\delta(\Ga)\in [2,3]$. Our results do not cover this example since the cusp is not full rank, but it would be very interesting to see a generalization of the results to lower degrees.\\

By carefully examining all these examples, we believe that there might be interesting general relations between the value of critical exponent and the number of cusps for a Kleinian group.

\begin{conjecture}
	Given a finitely generated Kleinian group $\Ga<\operatorname{Isom}^+(\mathbb H^{n+1})$, if $\delta(\Ga)<k$ for a positive integer $k\leq n$, then $\Ga$ has only finitely many rank $k$ cusps.
\end{conjecture}

We now proceed to discuss possible extensions of our results.

\subsection*{Other coefficient modules} To possibly extend our main results, one can try to consider other $\Ga$-coefficient modules $V$. We believe this should be straightforward. First, to construct a cohomology classes in $H^n(\Ga_\xi,V)$, one takes the $V_{-\lambda}$ sections where $\lambda$ denotes the highest weight of the representation. Second, the absolute convergence of Eisenstein series follows a similar argument provided $\lambda>0$, which is automatic if $V$ is not the trivial representation. Finally, the computation of the intertwining operators in Section \ref{sec:intertwining} works verbatim, and the argument in Section \ref{sec:proof} also works very similarly. Our Theorem $\ref{thm:Ad}$ is thus expected to hold for any non-trivial coefficient module. However, one cautionary point is that our Lemma \ref{lem:lie alg cohomology} may not hold any more.

\subsection*{Other degrees} One can also consider lower degree Eisenstein cohomology classes in $H^k(\Ga,V)$. But the construction of cohomology classes in $H^k(\Ga_\xi,V)$ is more delicate. For example, the Lie algebra cohomology $H^k(\mathfrak u, V)$ might be zero. When it is non-zero, then one can construct the Eisenstein series and it is absolutely convergent if the weights of the constructed cohomology classes are small enough (negative large) in terms of $k$ and the critical exponent of $\Ga$. The rest computation of the intertwining operators follows similarly. It would be very interesting to see if one can construct an absolutely convergent Eisenstein cohomology class in $H^1(\Ga,V)$. This would lead to a cusp finiteness theorem since finitely generated groups always satisfy $\dim H^1(\Ga,V)<\infty$.

\subsection*{Rank $k$ cusps} When $\Ga_\xi$ is a rank $k$ parabolic group, there are natural classes in $H^k(\Ga_\xi,V)$. Under the same construction, the Eisenstein series is comparable to the Poincar\'e series if the parabolic fixed point is \emph{bounded}. Thus, the absolute convergence again depends on the weights of the class constructed in $H^k(\Ga_\xi,V)$, the value $k$, and the critical exponent of $\Ga$. However, the most difficult part is the computation of the intertwining operators. Our method fails when $\phi_{\xi,s}$ is not $N_\xi$-invariant. One needs to further analyze the behavior of the function $\phi_{\xi,s}$ in the orthogonal directions to the subspace where it is invariant.

\bibliographystyle{alpha}
\bibliography{biblio}

\end{document}